\documentclass[reqno, 11pt]{amsart}
\usepackage{amssymb}
\usepackage[usenames, dvipsnames]{color}

\usepackage{verbatim}

\newcommand{\mysection}[1]{\section{#1}
\setcounter{equation}{0}}

\newtheorem{theorem}{Theorem}[section]
\newtheorem{corollary}[theorem]{Corollary}
\newtheorem{lemma}[theorem]{Lemma}
\newtheorem{proposition}[theorem]{Proposition}

\theoremstyle{definition}
\newtheorem{remark}[theorem]{Remark}

\theoremstyle{definition}

\theoremstyle{definition}
\newtheorem{assumption}[theorem]{Assumption}

\makeatletter
\def\dashint{\operatorname%
{\,\,\text{\bf--}\kern-.98em\DOTSI\intop\ilimits@\!\!}}
\makeatother

\def\vc{\textit{\textbf{c}}}
\def\vu{\textit{\textbf{u}}}
\def\vv{\textit{\textbf{v}}}
\def\vw{\textit{\textbf{w}}}
\def\vf{\textit{\textbf{f}}}
\def\vg{\textit{\textbf{g}}}
\def\vh{\textit{\textbf{h}}}
\def\vP{\textit{\textbf{P}}}

\def\rA{{\sf A}}
\def\rB{{\sf B}}

\def\bA{\mathbb{A}}
\def\bO{\mathbb{O}}
\def\bR{\mathbb{R}}
\def\bZ{\mathbb{Z}}

\def\bH{\mathbb{H}}

\def\bC{\mathbb{C}}

\def\cA{\mathcal{A}}
\def\cB{\mathcal{B}}
\def\cC{\mathcal{C}}
\def\cD{\mathcal{D}}

\def\cH{\mathcal{H}}

\def\cM{\mathcal{M}}

\def\cQ{\mathcal{Q}}

\def\cT{\mathcal{T}}
\def\cL{\mathcal{L}}

\def\cI{\mathcal{I}}

\newcommand{\Div}{\operatorname{div}}
\newcommand{\dist}{\text{dist}}
\begin{document}
\title[Higher order elliptic and parabolic systems]{Higher order elliptic and parabolic systems with variably partially BMO coefficients in regular and irregular domains}

\author[H. Dong]{Hongjie Dong}
\address[H. Dong]{Division of Applied Mathematics, Brown University,
182 George Street, Providence, RI 02912, USA}
\email{Hongjie\_Dong@brown.edu}
\thanks{H. Dong was partially supported by a start-up funding from the Division of Applied Mathematics of Brown University and NSF grant number DMS-0800129.}

\author[D. Kim]{Doyoon Kim}
\address[D. Kim]{Department of Applied Mathematics, Kyung Hee University, 1, Seochun-dong, Gihung-gu, Yongin-si, Gyeonggi-do 446-701 Korea}
\email{doyoonkim@khu.ac.kr}

\subjclass[2010]{35K52, 35J58,35R05}

\keywords{higher-order systems, vanishing mean oscillation, partially small BMO coefficients, Sobolev spaces.}

\begin{abstract}
The solvability in Sobolev spaces is proved for divergence form complex-valued higher order parabolic systems in the whole space, on a half space, and on a Reifenberg flat domain. The leading coefficients are assumed to be merely measurable in one spacial direction and have small mean oscillations in the orthogonal directions on each small cylinder.
The directions in which the coefficients are only measurable vary depending on each cylinder. The corresponding elliptic problem is also considered.
\end{abstract}

\maketitle

\setcounter{tocdepth}{1}
\tableofcontents

\mysection{Introduction}

We study the solvability in Sobolev spaces for parabolic operators in divergence form of order $2m$
\begin{equation}							 \label{eq0617_02}
\vu_t + (-1)^m \cL \vu
:=\vu_t+(-1)^m\sum_{|\alpha|\le m,|\beta|\le m}D^\alpha(a_{\alpha\beta}D^\beta\vu),
\end{equation}
where $\alpha$ and $\beta$ are multi-indices,
$a_{\alpha\beta}=[a_{\alpha\beta}^{ij}(x)]_{i,j=1}^n$ are $n\times n$ complex matrix-valued functions,
and $\vu$ is a complex vector-valued function.
As usual, for $\alpha=(\alpha_1,\ldots,\alpha_d)$, we write
$$
D^\alpha \vu =D_1^{\alpha_1}\ldots D_d^{\alpha_d} \vu.
$$
All the coefficients are assumed to be bounded and measurable, and $\cL$ is uniformly elliptic; cf. \eqref{eq11.28}.
In the case that  all the coefficients and functions involved are independent of the time variable, we also deal with elliptic operators in divergence form of order $2m$ as in \eqref{eq0617_02} without the $\vu_t$ term.

Our first focus in this paper is to find minimal regularity assumptions on coefficients for the $L_p$-solvability of {\em higher order} elliptic and parabolic systems.
In other words, we prove that there exist unique solutions in Sobolev spaces (e.g., $W_p^m(\Omega)$, $1<p<\infty$ in the elliptic case) to higher order elliptic and parabolic systems with coefficients having less regularity assumptions than those available in the literature.
We call the class of coefficients in the paper {\em variably partially BMO coefficients}, the key feature of which is that, on each cylinder (or ball in the elliptic case), the coefficients are allowed to have no regularity assumptions (i.e., merely measurable) in one spatial direction.
If we name the spatial direction {\em the measurable direction}, variably partially VMO coefficients mean that, for each small cylinder, there exists a measurable direction, which may depend only on the cylinder, such that coefficients are allowed to be merely measurable in that direction and have small mean oscillations in the orthogonal directions.
See Section \ref{sec082001} for a precise formulation of the assumption.
This class of coefficients was first introduced by N. V. Krylov about two years ago in \cite{Krylov08} in the context of second order elliptic equations in non-divergence form. The same class of coefficients is studied in \cite{DK09} for second order elliptic and parabolic systems in divergence form defined in the whole space. See also a more recent paper \cite{BW10} for second order elliptic equations without lower-order terms and with a symmetric coefficient matrix on a bounded Reifenberg flat domain.

In order to put into perspective the class of coefficients in this paper,
let us review very briefly the coefficients for $L_p$-theory of elliptic and parabolic equations/systems studied in the literature.
A rather rich collection of papers about $L_p$-theory is devoted to the study of equations with {\em constant or uniformly continuous coefficients}.
For instance, see \cite{ADN64,A65,Mi06} for elliptic systems and
\cite{Solo,LSU,Fried} for parabolic systems.
There are also a large number of papers concerning equations with discontinuous coefficients.
An important example of discontinuous coefficients is a family of functions with vanishing mean oscillations (VMO), which was firstly considered in \cite{CFL2, BC93, DFG} for second order elliptic and parabolic equations. With this class of coefficients, $L_p$-estimates for higher order equations have been obtained in \cite{CFF}, \cite{HHH}, \cite{PS1}, \cite{PS3} and \cite{MaMiSh}. We also refer the reader to \cite{Lo72} for elliptic equations with piecewise constant coefficients, and \cite{Davies} for results of elliptic equations with measurable coefficients for a restricted range of $m$ or $p$, and \cite{Fried83,AuQa} for H\"older estimates of systems with uniformly continuous or VMO coefficients, as well as \cite{B09} for a result of fourth order elliptic equations in nonsmooth domains.

Recently, in \cite{DK09_01} we considered higher order systems in both divergence and non-divergence form with coefficients having locally small mean oscillations with respect to the spatial variables (and measurable in the time variable in the parabolic case).
We call coefficients in \cite{DK09_01} BMO coefficients, the class of which includes VMO coefficients as a proper subset.
In this paper the coefficients also have small mean oscillations, but need no regularity assumptions in one {\em spatial} direction (determined locally). Thus especially in the elliptic case the coefficients here are strictly more general than those studied in \cite{DK09_01}.
For the parabolic case, once a measurable direction is determined on each cylinder, the coefficients are required to have small mean oscillations in the orthogonal directions, which include the time direction,
whereas the coefficients in \cite{DK09_01} are merely measurable in time, but have small mean oscillations in all the spatial directions.
Thus the two classes of coefficients here and in \cite{DK09_01} are not mutually inclusive.

Our second focus in this paper is to examine the $L_p$-solvability of elliptic and parabolic systems {\em in various domains}.
We first deal with systems in the whole space, where interior $L_p$-estimates are the key ingredients. Then we prove the solvability of systems defined on a half space, where we obtain boundary $L_p$-estimates.
As usual, combining the interior and boundary $L_p$-estimates yields the solvability of systems defined on a domain as long as the boundary is regular, for instance, if the domain is a Lipschitz domain with a small Lipschitz constant.
Indeed, in \cite{DK09_01} elliptic and parabolic systems with BMO coefficients are studied in the whole space, on a half space, and on a Lipschitz domain.
Here we further examine the $L_p$-solvability of systems with variably partially BMO coefficients defined on Reifenberg flat domains (possibly unbounded), which are more general than Lipschitz domains.
As is well known, at small scales the boundary of a Reifenberg flat domain can be approximated by hyper-planes.
We remark that, for the Reifenberg flat domain case, on each small cylinder centered at the boundary, we assume that coefficients are merely measurable in the direction normal to the approximating hyper-plane.

In the half space case, compared to previous results in \cite{KimKrylov07,KimKrylov:par06,DongKim08a} for second order equations, one of the virtues of our results is that
on each small cylinder centered on the boundary, we do not require the measurable direction to be exactly perpendicular to the boundary, but allow it to be sufficiently close to the normal direction. To this end, our proof is founded on a delicate cutoff argument combined with a higher order Hardy-type inequality; see Lemma \ref{gHardy}.
Indeed, due to the nature of the boundary of Reifenberg flat domains, studying the half space case with almost normal measurable directions is readily applied to obtaining the key estimates for systems on Reifenberg flat domains with variably partially BMO coefficients.
For further discussions about equations/systems on Reifenberg flat domains, see Section \ref{Reifenberg}.


The results in this paper and in \cite{DK09_01} imply several generalizations of the interior H\"older estimate for higher order elliptic equations with VMO coefficients proved in \cite{AuQa} Proposition 50. Indeed, by the Sobolev imbedding theorem, we obtain {\em both interior and boundary} H\"older estimates for higher order elliptic and parabolic systems with coefficients in \cite{DK09_01} and with coefficients considered in this paper.

It is worth mentioning that we impose the uniform ellipticity condition in this paper, while the main results in \cite{DK09_01} are established under a slightly weaker Legendre-Hadamard ellipticity condition. Because the coefficients are measurable in one spatial variable, when getting the $L_2$-estimates (or equivalently the G{\aa}rding inequality) the Fourier transform method used in \cite{DK09_01} is not applicable here.
This is the only place where the uniform ellipticity condition plays its role.

Our proofs for the whole space and half space cases are in the spirit of \cite{Krylov_2005}, in which the author gave a novel method of studying $L_p$ estimates of second order elliptic and parabolic equations in the whole space with rough coefficients.
Differed from the arguments in \cite{CFL2,DFG,CFF,PS3}, which are based on singular integrals and commutator estimates, the crucial step in \cite{Krylov_2005} is to establish certain interior {\em mean oscillation estimates} of solutions to equations with `simple' coefficients, which are measurable in some directions and independent of the others. As a consequence, VMO coefficients are treated in a rather straightforward way. The method in \cite{Krylov_2005} was later further developed in a series of papers
\cite{Krylov_2007_mixed_VMO,KimKrylov07,KimKrylov:par06,Kim:par06,Kim07a,Krylov08,DongKrylov08,DK09,Dong08b} on $L_p$-estimates of second order equations with rather general discontinuous coefficients.

One of the main differences in our study between second order equations and higher order equations is the following. As is explained in \cite{KimKrylov07,KimKrylov:par06,DongKim08a},
when dealing with second order equations (or systems) on a half space with homogeneous Dirichlet or Neumann boundary conditions,
we used the odd and even extension techniques by observing that the odd or even extension of the solution satisfies the equation extended in the whole space.
Then we were able to rely on the fact that the coefficients are allowed to be merely measurable in one spatial direction. Thus the boundary $L_p$-estimates were derived  almost immediately from interior estimates.
However, the extension techniques do not work for higher order equations or systems, since in general extensions over the whole space do not satisfy the extended equations and,
in the non-divergence case, they do not even belong to the correct solution spaces. Because of this, in this paper we use a modified version of a generalized Fefferman-Stein Theorem developed in \cite{Krylov08} in order to produce {\em boundary} mean oscillation estimates of solutions.
This approach for boundary  $L_p$-estimates is applicable to a wide class of equations or systems.

Another worth noting difference comparing to second order equations is in the proof of interior H\"older estimates; see Lemmas \ref{lemma01} and \ref{lemma3.6}. To get around the lack of regularity of $D_1\vu$, the main idea in \cite{DK09} is to estimate instead the interior H\"older norm of certain linear combination of derivatives of the solution which in particular contains $D_1\vu$. For higher order systems, the proof is more involved and we also make use of some $L_p$ estimate of a one dimensional problem; see Lemma \ref{lemA.1}.


A brief outline of the paper: We introduce some notation and state the main theorems in the next section. Section \ref{sec_aux} is devoted to some auxiliary results including the $L_2$-estimates. In Section \ref{sec3.2} we prove some H\"older estimates. We obtain the interior mean oscillation estimates in Section \ref{sec4} and prove the solvability of systems in the whole space, i.e. Theorem \ref{thm1}, in Section \ref{sec5}. In Section \ref{sec6}, we establish several boundary estimates including the boundary mean oscillation estimates and prove the solvability of systems on a half space, i.e. Theorem \ref{thm3}.
Finally we deal with parabolic and elliptic systems on a Reifenberg flat domain in Section \ref{Reifenberg}.

\mysection{Main results}								 \label{sec082001}

Before we state our assumptions and main theorems, we introduce some necessary notation.
By $\bR^d$ we mean a $d$-dimensional Euclidean space,
a point in $\bR^d$ is denoted by $x=(x_1,\ldots,x_d)=(x_1,x')$,
and $\{e_j\}_{j=1}^d$ is the standard basis of $\bR^d$.
Let $\bR^{d+1}:=\bR \times \bR^d = \{ (t,x) : t \in \bR, x \in \bR^d \}$.
Throughout the paper, $\Omega$ indicates an open set in $\bR^d$
and $\Omega_T := (-\infty,T)\times\Omega \subset \bR^{d+1}$.
For vectors $\xi,\eta\in \bC^n$, we denote
$$
(\xi,\eta)=\sum_{i=1}^n \xi^i\overline{\eta^i}.
$$
For a function $f$ defined on a subset $\cD$ in $\bR^{d+1}$, we set
\begin{equation*}
(f)_{\cD} = \frac{1}{|\cD|} \int_{\cD} f(t,x) \, dx \, dt
= \dashint_{\cD} f(t,x) \, dx \, dt,
\end{equation*}
where $|\cD|$ is the
$d+1$-dimensional Lebesgue measure of $\cD$.

Throughout the paper, we assume that the $n \times n$ complex-valued coefficient matrices $a_{\alpha\beta}$ are measurable and bounded, and the leading coefficients $a_{\alpha\beta}$, $|\alpha|=|\beta|=m$, satisfy an ellipticity condition.
More precisely, we assume the following.
\begin{enumerate}
\item There exists a constant $\delta \in (0,1)$ such that
the leading coefficients $a_{\alpha\beta}$, $|\alpha|=|\beta|=m$, are bounded by $\delta^{-1}$
and satisfy
\begin{equation}
                            \label{eq11.28}
\delta |\xi|^2 \le \sum_{|\alpha|=|\beta|=m}\Re(a_{\alpha\beta}(t,x) \xi_{\beta}, \xi_{\alpha})
\le \delta^{-1} |\xi|^2
\end{equation}
for any $(t,x)\in \bR^{d+1}$ and $\xi = (\xi_{\alpha})_{|\alpha|=m}$, $\xi_{\alpha} \in \bC^n$.
Note that $\xi$ can be considered as a vector in $\bC^{n\times {m+d-1\choose d-1}}$
where
$$
{m+d-1\choose d-1} = \sum_{|\alpha|=m}1 = \frac{(m+d-1)!}{m!(d-1)!}.
$$
Here we use $\Re(f)$ to denote the real part of $f$.

\item All the lower-order coefficients $a_{\alpha\beta}$, $|\alpha| \ne m$ or $|\beta| \ne m$, are bounded by a constant $K\ge 1$.
\end{enumerate}

Let
$$
B_r(x) = \{ y \in \bR^d: |x-y| < r\},
\quad
B'_r(x') = \{ y' \in \bR^{d-1}: |x'-y'| < r\},
$$
$$
Q_r(t,x) = (t-r^{2m},t)\times B_r(x),
\quad
Q'_r(t,x') = (t-r^{2m},t)\times B'_R(x').
$$

On the leading coefficients we impose a very mild regularity assumption with a parameter $\gamma\in (0,1)$, which will be specified later.
To state this assumption, throughout the paper we write $\{\bar{a}_{\alpha\beta}\}_{|\alpha|=|\beta|=m} \in \bA$ whenever the
$n \times n$ complex-valued matrices $\bar{a}_{\alpha\beta}$ are measurable functions of $y_1 \in \bR$ only, $|\bar{a}_{\alpha\beta}|\le \delta^{-1}$,
and $\{\bar{a}_{\alpha\beta}\}_{|\alpha|=|\beta|=m}$ satisfies the ellipticity condition \eqref{eq11.28}.
For a linear map $\cT$ from $\bR^d$ to $\bR^d$, we write $\cT \in \bO$ if
$\cT$ is of the form
$$
\cT(x) = \rho x + \xi,
$$
where $\rho$ is a $d \times d$ orthogonal matrix
and $\xi \in \bR^d$.

\begin{assumption}[$\gamma$]                          \label{assumption20080424}
There is a constant $R_0\in (0,1]$ such that,
for each parabolic cylinder $Q:=(t_0-r^{2m},t_0)\times B_r(x_0)$ with $r \le R_0$,
one can find $\cT_Q \in \bO$
and coefficient matrices $\{\bar a_{\alpha\beta}\}_{|\alpha|=|\beta|=m} \in \bA$ satisfying
\begin{equation}
								\label{eq10_23}
\sup_{|\alpha|=|\beta|=m}\int_Q |a_{\alpha\beta}(t,x) - \bar{a}_{\alpha\beta}(y_1)| \, dx \, dt \le \gamma |Q|,
\end{equation}
where $y = \cT_Q(x)$.
\end{assumption}

Let us introduce some function spaces utilized throughout the paper.
We define the solution spaces $\cH_p^m((S,T)\times\Omega)$ as follows.
For a domain $\Omega$ in $\bR^d$,
we set
$$
\bH^{-m}_p((S,T)\times \Omega)
= \Big\{ f: f = \sum_{|\alpha|\le m} D^{\alpha}f_{\alpha}, \quad f_{\alpha} \in L_p((S,T) \times \Omega)\Big\},
$$
$$
\|f\|_{\bH^{-m}_p((S,T)\times \Omega)}
= \inf \Big\{ \sum_{|\alpha|\le m} \|f_{\alpha}\|_{L_p((S,T)\times \Omega)} : f = \sum_{|\alpha|\le m} D^{\alpha}f_{\alpha}\Big\}.
$$
Then
\begin{multline*}
\cH_p^m((S,T) \times \Omega)\\
=\{u: u_t \in \bH_p^{-m}((S,T)\times\Omega), D^{\alpha}u \in L_p((S,T)\times\Omega), 0 \le |\alpha| \le m \},
\end{multline*}
$$
\|u\|_{\cH_p^m((S,T) \times \Omega)}
= \|u_t\|_{\bH_p^{-m}((S,T)\times\Omega)} + \sum_{|\alpha|\le m} \|D^{\alpha}u\|_{L_p((S,T)\times\Omega)}.
$$
In this paper $\vu \in C_{\text{loc}}^{\infty}(\cD)$ means that $\vu$ is infinitely differentiable on $\cD$, where $\cD$ is a subset of either $\bR^{d+1}$ or $\bR^d$.
As usual, $C_0^{\infty}(\cD)$ means the collection of infinitely differentiable functions with compact support $\Subset \cD$.
We define $C_0^\infty([S,T]\times\Omega)$ to be the collection of infinitely differentiable functions $\phi(t,x)$ defined on $[S,T]\times\Omega$ such that, for each $t \in [S,T]$, $\phi(t,\cdot) \in C_0^{\infty}(\Omega)$.
The reader understands that if either $S$ or/and $T$ is infinity, then the closed interval should be replaced by a half-open or open interval.
Finally, we set $\mathring \cH^m_p((S,T)\times\Omega)$ to be the closure of $C_0^\infty([S,T]\times\Omega)$ in $\cH^m_p((S,T)\times\Omega)$.

Now we state the main result concerning parabolic systems in divergence form defined in the whole space.

\begin{theorem}
    							\label{thm1}
Let $\Omega=\bR^d$, $p \in (1,\infty)$, $T\in (-\infty,+\infty]$,
and $\vf_\alpha= (f_\alpha^1, \ldots, f_\alpha^n)^{\text{tr}} \in L_p(\Omega_T)$, $|\alpha|\le m$.
Then there exists a constant $\gamma=\gamma(d,n,m,p,\delta)$
such that, under Assumption \ref{assumption20080424} ($\gamma$),
the following hold true.

\noindent
(i)
For any $\vu \in \cH_p^m(\Omega_T)$ satisfying
\begin{equation}							 \label{eq081902}
\vu_t + (-1)^m \cL \vu + \lambda \vu = \sum_{|\alpha|\le m}D^\alpha \vf_\alpha
\end{equation}
in $\Omega_T$,
we have
\begin{equation}							 \label{eq080904}
\sum_{|\alpha|\le m}\lambda^{1-\frac {|\alpha|} {2m}} \|D^\alpha \vu \|_{L_p(\Omega_T)}
\le N \sum_{|\alpha|\le m}\lambda^{\frac {|\alpha|} {2m}} \| \vf_\alpha \|_{L_p(\Omega_T)},
\end{equation}
provided that $\lambda \ge \lambda_0$,
where $N$ and $\lambda_0 \ge 0$
depend only on $d$, $n$, $m$, $p$, $\delta$, $K$ and $R_0$.

\noindent
(ii)
For any  $\lambda > \lambda_0$, there exists a unique $\vu \in \cH^m_p(\Omega_T)$ satisfying \eqref{eq081902}.

\noindent
(iii)
If all the lower-order coefficients of $\cL$ are zero and the leading coefficients are measurable functions of $x_1 \in \bR$ only, then one can take $\lambda_0=0$.
\end{theorem}

Our next result is about the Dirichlet problem on a half space.
For this, we impose the following assumption, where the parameter $\gamma \in (0,1/4)$ is
to be determined later.
Set $\bR^d_+ = \{(x_1,x') \in \bR^d: x_1 > 0\}$ and $\bR^{d+1}_+=\bR\times\bR^d_+$.

\begin{assumption}[$\gamma$]                          \label{assumption20100901}
There is a constant $R_0\in (0,1]$ such that the following holds with  $Q:=(t_0-r^{2m},t_0)\times B_r(x_0)$.

i) For any $x\in \bR^d_+$, $t\in \bR$ and any $r\in \left(0,\min\{R_0,\dist(x,\partial\Omega)\}\right]$ so that $Q\subset \bR^{d+1}_+$,
one can find $\cT_Q \in \bO$
and coefficient matrices $\{\bar a_{\alpha\beta}\}_{|\alpha|=|\beta|=m} \in \bA$ satisfying \eqref{eq10_23}.

ii) For any $x\in \partial \bR^d_+$, $t\in \bR$ and any $r\in (0,R_0]$, one can find $\cT_Q \in \bO$ satisfying $\rho_{11}\ge \cos (\gamma/2)$
and coefficient matrices $\{\bar a_{\alpha\beta}\}_{|\alpha|=|\beta|=m} \in \bA$ satisfying \eqref{eq10_23}.
\end{assumption}

With a sufficiently small $\gamma$, the condition $\rho_{11}\ge \cos (\gamma/2)$ means that at any boundary point the $y_1$-direction is sufficiently close to the $x_1$-direction, i.e., the normal direction of the boundary.

\begin{theorem}
                                    \label{thm3}
Let $\Omega=\bR^d_+$, $p \in (1,\infty)$, $T\in (-\infty,+\infty]$
and
$$
\vf_\alpha= (f_\alpha^1, \ldots, f_\alpha^n)^{\text{tr}} \in L_p(\Omega_T), \quad |\alpha|\le m.
$$
Then there exists a constant $\gamma=\gamma(d,n,m,p,\delta)$
such that, under Assumption \ref{assumption20100901} ($\gamma$),
the following hold true.

\noindent
(i)
For any $\vu \in \mathring \cH^m_p(\Omega_T)$ satisfying
\begin{equation}
                                    \label{eq1.55}
\vu_t+(-1)^m\cL \vu +\lambda \vu = \sum_{|\alpha|\le m}D^\alpha \vf_\alpha
\end{equation} in $\Omega_T$,
we have
\begin{equation*}							 
\sum_{|\alpha|\le m}\lambda^{1-\frac {|\alpha|} {2m}} \|D^\alpha \vu \|_{L_p(\Omega_T)}
\le N \sum_{|\alpha|\le m}\lambda^{\frac {|\alpha|} {2m}} \| \vf_\alpha \|_{L_p(\Omega_T)},
\end{equation*}
provided that $\lambda \ge \lambda_0$,
where $N$ and $\lambda_0 \ge 0$
depend only on $d$, $n$, $m$, $p$, $\delta$, $K$ and $R_0$.

\noindent
(ii)
For any  $\lambda > \lambda_0$, there exists a unique $\vu \in \mathring \cH_p^m(\Omega_T)$ satisfying \eqref{eq1.55}.

\noindent
(iii)
If all the lower-order coefficients of $\cL$ are zero and the leading coefficients are measurable functions of $x_1\in \bR$ only, then one can take $\lambda_0=0$.
\end{theorem}

Solutions of \eqref{eq1.55} or \eqref{eq09_01} below are understood in the weak sense: we say $\vu \in \cH_p^m((S,T)\times\Omega)$ satisfies \eqref{eq1.55} in $(S,T)\times\Omega$ if
$$
\int_{\Omega} \phi(T,\cdot)\cdot \vu(T,\cdot)\,dx
-\int_S^T\int_{\Omega}\phi_t\cdot u\,dx\,dt
$$
$$
+\sum_{|\alpha|\le m,|\beta|\le m}
\int_S^T\int_{\Omega}\left((-1)^{m+|\alpha|}
D^\alpha\phi \cdot a_{\alpha\beta}D^\beta\vu +\lambda \phi\cdot \vu\right)\,dx\,dt
$$
$$
=\sum_{|\alpha|\le m}\int_S^T\int_{\Omega}(-1)^{|\alpha|}D^\alpha\phi\cdot \vf_\alpha\,dx\,dt
+ \int_{\Omega}\phi(S,\cdot) \cdot \vu(S,\cdot) \, dx
$$
for any test function $\phi=(\phi^1,\phi^2,\ldots,\phi^n)\in C_0^\infty([S,T]\times\Omega)$.
If $S = -\infty$ or $T = \infty$, we assume $\phi(-\infty,\cdot) = 0$ or $\phi(\infty,\cdot) = 0$, respectively.

\begin{remark}
                                \label{rm2.4}
The ellipticity condition \eqref{eq11.28} can be relaxed to a weaker condition
\begin{equation}
                            \label{eq22.21}
\delta \sum_{j=1}^d|\xi_{me_j}|^2\le
\sum_{|\alpha|=|\beta|=m}
\Re(a_{\alpha\beta}(t,x)\xi_\beta, \xi_\alpha)\le \delta^{-1} |\xi|^2.
\end{equation}
For instance, when $d=m=2$ the operator $\cL=D_1^4+D_2^4$ satisfies \eqref{eq22.21} with $\delta=1$, but is not elliptic in the sense of \eqref{eq11.28}. However, the condition \eqref{eq11.28} has the advantage that it is invariant under orthogonal transformations of the coordinates. We claim that any operator $\cL$ satisfying \eqref{eq22.21} can be rewritten into another divergence form operator which satisfies \eqref{eq11.28} with a possibly different $\delta$. In the above example, one way is to write
$$
D_1^4 u+D_2^4 u=D_1^4 u+D_2^4 u-D_1^2(D_2^2 u)+D_{12}(D_{12} u).
$$
The symbol of the right-hand side is
$$
\xi_{(2,0)}^2+ \xi_{(0,2)}^2-\xi_{(2,0)}\xi_{(0,2)}+\xi_{(1,1)}^2,
$$
which obviously satisfies \eqref{eq11.28} with $\delta=1/2$.

The claim is a simple consequence of the following observation. We only consider the case $d=2$. The general case follows from an induction using linear interpolations to cover the convex hull of $d$ vertices.
For simplicity, we also assume $n=1$ and everything is real. First, it is easy to check that
\begin{multline*}
\sum_{j=0}^m 2^{j^2}\xi_{(j,m-j)}^2-\sum_{j=1}^{m-1}2^{j^2}\xi_{(j-1,m-j+1)}\xi_{(j+1,m-j-1)}\\
\ge \sum_{j=0}^m 2^{j^2-1}\xi_{(j,m-j)}^2
+\sum_{j=1}^{m-1}\left(2^{(j-1)^2-2}\xi_{(j-1,m-j+1)}^2
+2^{(j+1)^2-2}\xi_{(j+1,m-j-1)}^2\right)\\
-\sum_{j=1}^{m-1}2^{j^2}\xi_{(j-1,m-j+1)}\xi_{(j+1,m-j-1)}
\ge  \sum_{j=0}^m 2^{j^2-1}\xi_{(j,m-j)}^2.
\end{multline*}
Therefore, there exist $\varepsilon=\varepsilon(m)>0$ and $\delta_1=\delta_1(m)>0$ such that
$$
\sum_{j=1}^2|\xi_{me_j}|^2+\varepsilon\sum_{j=1}^{m-1} \left(2^{j^2}\xi_{(j,m-j)}^2-2^{j^2}\xi_{(j-1,m-j+1)}\xi_{(j+1,m-j-1)}
\right)\ge \delta_1 |\xi|^2.
$$
Then suppose $\cL$ satisfies \eqref{eq22.21}.
Using the fact that
$$
D^{(j,m-j)}D^{(j,m-j)}=D^{(j-1,m-j+1)}D^{(j+1,m-j-1)},
$$
we then rewrite $\cL$ as
$$
\cL+\varepsilon \sum_{j=1}^{m-1} \left(2^{j^2}D^{(j,m-j)}D^{(j,m-j)}-2^{j^2}D^{(j-1,m-j+1)}D^{(j+1,m-j-1)}
\right),
$$
which satisfies \eqref{eq11.28} with $\delta\delta_1$ in place of $\delta$. This completes the proof of the claim. Note that the leading coefficients of the new operator satisfy the same regularity assumption as those of $\cL$. Therefore, the results of our main theorems still hold true under the condition \eqref{eq22.21}.
\end{remark}

\mysection{Some auxiliary estimates}
\label{sec_aux}

In this section we consider operators without lower-order terms. Denote
$$
\cL_0 \vu = \sum_{|\alpha|=|\beta|=m}D^\alpha( a_{\alpha\beta} D^\beta \vu).	
$$

\subsection{$L_2$-estimates}                    \label{sec3.1}
The first result is the classical $L_2$-estimate for parabolic operators in divergence form with measurable coefficients. We give a sketched proof for the sake of completeness.

\begin{theorem}			\label{theorem08061901}
Let $T\in (-\infty,\infty]$ and $\Omega = \bR^d$.
There exists $N = N(d,m,n, \delta)$
such that, for any $\lambda \ge 0$,
\begin{equation}
								\label{eq2010_01}
\sum_{|\alpha|\le m}\lambda^{1-\frac {|\alpha|} {2m}} \|D^\alpha \vu \|_{L_2(\Omega_T)}
\le N \sum_{|\alpha|\le m}\lambda^{\frac {|\alpha|} {2m}} \| \vf_\alpha \|_{L_2(\Omega_T)},
\end{equation}
provided that $\vu \in \cH_2^m(\Omega_T)$, $\vf_\alpha \in L_2(\Omega_T),|\alpha|\le m$,
and
\begin{equation}							 \label{eq080501}
\vu_t + (-1)^m\cL_0 \vu + \lambda \vu = \sum_{|\alpha|\le m}D^\alpha \vf_\alpha
\end{equation}
in $\Omega_T$.
Furthermore, for any $\lambda > 0$ and $\vf_\alpha \in L_2(\Omega_T),|\alpha|\le m$, there exists a unique solution $\vu\in \cH_2^m(\Omega_T)$ to
the equation \eqref{eq080501}.
\end{theorem}
\begin{proof}
By the method of continuity and a standard density argument, it suffices to prove the estimate \eqref{eq2010_01} for $u \in C_0^{\infty}((-\infty,T]\times \Omega)$.
From the equation, it follows that
$$
\int_{\Omega_T} \left[(\vu,\vu_t)
+ (D^{\alpha}\vu, a_{\alpha\beta}D^{\beta}\vu)
+ \lambda |\vu|^2 \right]\,dx\,dt
$$
\begin{equation}
                            \label{eq16.33}
= \sum_{|\alpha|\le m} (-1)^{|\alpha|} \int_{\Omega_T} (D^{\alpha}\vu, f_{\alpha}) \, dx \, dt.
\end{equation}
By the uniform ellipticity \eqref{eq11.28}, we get
$$
\delta \int_{\Omega_T} |D^m\vu|^2 \, dx \, dt
\le \int_{\Omega_T} \Re(a_{\alpha\beta} D^\beta \vu, D^\alpha \vu) \, dx \, dt.
$$
We also have
$$
\int_{\Omega_T} \Re (\vu,\vu_t) \, dx \, dt
= \frac12 \int_{\bR^d} |\vu|^2(T,x) \, dx \ge 0.
$$
Hence, for any $\varepsilon>0$,
$$
\delta \int_{\Omega_T} |D^m\vu|^2 \, dx \, dt + \lambda \int_{\Omega_T} |\vu|^2 \, dx\,dt
\le \sum_{|\alpha|\le m}(-1)^{|\alpha|} \int_{\Omega_T}\Re(D^\alpha \vu, \vf_\alpha) \, dx \, dt
$$
$$
\le \varepsilon \sum_{|\alpha|\le m}\lambda^{\frac {m-|\alpha|} m} \int_{\Omega_T} |D^\alpha \vu|^2 \, dx \, dt  + N\varepsilon^{-1} \sum_{|\alpha|\le m}\lambda^{-\frac {m-|\alpha|} m}\int_{\Omega_T} |\vf_\alpha|^2 \, dx \, dt.
$$
To finish the proof, it suffices to use interpolation inequalities
and choose $\varepsilon$ sufficiently small depending on $\delta,d,m$ and $n$.
\end{proof}

We also have the following local $L_2$-estimate, where for a later use we consider only a simple case that the right-hand side of \eqref{eq2010_02} is zero.

\begin{lemma}
                                            \label{lem2010_03}
Let $0<r<R<\infty$.
Assume $\vu \in C_{\text{loc}}^{\infty}(\bR^{d+1})$ and
\begin{equation}
								\label{eq2010_02}
\vu_t +(-1)^m\cL_0 \vu = 0
\end{equation}
in $Q_R$. Then there exists a constant $N=N(d,m,n,\delta)$ such that for $j=1,2,\ldots,m$,
\begin{equation}
                                          \label{eq2010_10}
\|D^j\vu\|_{L_2(Q_r)}\leq N(R-r)^{-j}\|\vu\|_{L_2(Q_R)}.
\end{equation}
\end{lemma}
\begin{proof}
First we consider the case $j=m$. Set
$$
r_0 = r,
\quad
r_k = r +\sum_{i=1}^k\frac{R-r}{2^k},
\quad
k = 1, 2, \ldots,
$$
$$
s_k=\frac{r_k + r_{k+1}}2,\quad
k = 0, 1, 2, \ldots.
$$
We choose nonnegative real-valued $\zeta_k(t,x) \in C_0^{\infty}(\bR^{d+1})$ such that
$$
\zeta_k
= \left\{\begin{aligned}
1 \quad &\text{on} \quad Q_{r_k},\\
0 \quad &\text{on} \quad \bR^{d+1} \setminus (-s_k^{2m}, s_k^{2m}) \times B_{s_k},
\end{aligned}\right.
$$
and
\begin{equation}
								\label{eq10_07}
| (\zeta_k)_t | \le N \frac{2^{2mk}}{(R-r)^{2m}},
\quad
| D^l\zeta_k| \le N \frac{2^{lk}}{(R-r)^l},
\quad
l=0,1,\ldots,m.	
\end{equation}

By applying $\vu\zeta_k^2$ as a test function to the system \eqref{eq2010_02} we get
\begin{equation}
								\label{eq10_03}
\int_{Q_R}(\vu\zeta^2_k, \vu_t) \,dx\,dt
+ \int_{Q_R}(D^\alpha(\vu\zeta_k^{2}), a_{\alpha\beta}D^\beta\vu) \,dx\,dt = 0.
\end{equation}
Note that
$$
\int_{Q_R}(\vu\zeta_k^{2}, \vu_t) \,dx\,dt
= \int_{B_R} |\vu\zeta_k|^2(0,x)\, dx
- \int_{Q_R} (\vu_t, \vu\zeta_k^{2}) \,dx\,dt
$$
$$
- \int_{Q_R}2\zeta_k(\zeta_k)_t|\vu|^2 \,dx\,dt,
$$
which shows that
\begin{equation}
								\label{eq10_04}
\Re \int_{Q_R} (\vu\zeta_k^{2},\vu_t) \,dx\,dt
= \frac12 \int_{B_R} |\vu\zeta_k|^2(0,x)\, dx - \int_{Q_R}\zeta_k(\zeta_k)_t|\vu|^2 \,dx\,dt.
\end{equation}
On the other hand, by Leibniz's rule
\begin{multline}
								\label{eq10_05}
\int_{Q_R}\left(D^\alpha(\vu\zeta_k^{2}),a_{\alpha\beta}D^\beta\vu\right) \,dx\,dt
= \int_{Q_R}\left(\zeta_k D^\alpha\vu, a_{\alpha\beta} \zeta_k D^\beta\vu\right) \,dx\,dt
\\
+ \int_{Q_R}\sum_{\substack{\alpha_1+\alpha_2=\alpha\\|\alpha_2|<m}}
c_{\alpha_1,\alpha_2}\left(D^{\alpha_1}\zeta_k^{2}\right)(D^{\alpha_2}\vu, a_{\alpha\beta}D^\beta\vu) \,dx\,dt := I_1 + I_2,
\end{multline}
where $\alpha_1, \alpha_2$ are multi-indices,
and $c_{\alpha_1,\alpha_2}$ are corresponding appropriate constants.
By the ellipticity condition \eqref{eq11.28}, it follows that
\begin{equation}
								\label{eq10_06}
\Re(I_1) = \int_{\bR^d_0} \Re \left(\zeta_k D^\alpha\vu, a_{\alpha\beta} \zeta_k D^\beta\vu\right) \,dx\,dt
\ge \delta \int_{\bR^d_0} |\zeta_k D^m\vu|^2\,dx\,dt.
\end{equation}
Here, we recall $\bR^d_0=(-\infty,0)\times\bR^d$.
To estimate $I_2$, we first see that $Q_R$ and $\vu$ in the integrals can be replaced by
$\bR^d_0$ and $\vu\zeta_{k+1}$, respectively.
Then using \eqref{eq10_07}, we have
$$
|I_2| \le N \sum_{l=0}^{m-1}\sum_{|\alpha_2|=l, |\beta|=m}\frac{2^{(m-l)k}}{(R-r)^{m-l}}\int_{\bR^d_0}|D^{\alpha_2}(\vu\zeta_{k+1})|
|D^{\beta}(\vu\zeta_{k+1})|\,dx\,dt := I_3.
$$
Set
$$
\rB = \|\vu\|_{L_2(Q_R)}^2.
$$
Combining \eqref{eq10_03}, \eqref{eq10_04}, \eqref{eq10_05}, and \eqref{eq10_06} as well as using the inequality for $(\zeta_k)_t$ in \eqref{eq10_07}, we obtain
\begin{equation}
								\label{eq10_09}
\delta \int_{\bR^d_0} |\zeta_k D^m\vu|^2 \,dx\,dt
\le
N\frac{2^{2mk}}{(R-r)^{2m}}\rB + I_3.	
\end{equation}
To estimate $I_3$, using Young's inequality we observe that,
for each $0 \le l \le m-1$,
$$
\int_{\bR^d_0} |D^l(\vu\zeta_{k+1})| |D^m(\vu\zeta_{k+1})|\,dx\,dt
$$
$$
\le \varepsilon\frac{(R-r)^{m-l}}{2^{(m-l)k}} \int_{\bR^d_0} |D^m(\vu\zeta_{k+1})|^2\,dx\,dt
+ \frac{2^{(m-l)k}}{4\varepsilon(R-r)^{m-l}}\int_{\bR^d_0}|D^l(\vu\zeta_{k+1})|^2\,dx\,dt,
$$
where $\varepsilon > 0$ is an arbitrary real number.
Furthermore, for $l=1,\ldots,m-1$, by interpolation inequalities
$$
\int_{\bR^d_0} |D^l(\vu\zeta_{k+1})|^2\,dx\,dt
\le \varepsilon_0 \int_{\bR^d_0} |D^m(\vu\zeta_{k+1})|^2\,dx\,dt
+ N\varepsilon_0^{\frac{l}{l-m}} \rB,
$$
where we set $\varepsilon_0 = 4\varepsilon^{2}(R-r)^{2(m-l)}2^{2(l-m)k}$.
Combining the above two inequalities with \eqref{eq10_09} implies
\begin{multline}
								\label{eq10_08}
\delta \int_{\bR^d_0} |\zeta_k D^m\vu|^2 \,dx\,dt \le \varepsilon \int_{\bR^d_0}|D^m(\vu\zeta_{k+1})|^2\,dx\,dt
\\
+N\left(1+ \sum_{l=0}^{m-1}\varepsilon^{\frac{l+m}{l-m}}\right)\frac{2^{2mk}}{(R-r)^{2m}}\rB.	 \end{multline}

Now we set
$$
\rA_k = \int_{\bR^d_0} |D^m(\vu\zeta_k)|^2 \,dx\,dt.
$$
To estimate $\rA_k$,
we use \eqref{eq10_07} and interpolation inequalities to get
$$
\int_{\bR^d_0} |\vu D^m\zeta_k|^2 \,dx\,dt
= \int_{Q_R}|\vu D^m\zeta_k|^2 \,dx\,dt
\le N\frac{2^{2mk}}{(R-r)^{2m}}\rB,
$$
and, for $1\le l \le m-1$,
$$
\int_{\bR^d_0} |D^{m-l}\zeta_k D^l \vu|^2 \, dx \, dt
=\int_{\bR^d_0} |D^{m-l}\zeta_k D^l (\vu\zeta_{k+1})|^2 \, dx \, dt
$$
$$
\le N \frac{2^{2(m-l)k}}{(R-r)^{2(m-l)}}\int_{\bR^d_0} |D^l (\vu\zeta_{k+1})|^2 \, dx \, dt
$$
$$
\le \varepsilon \int_{\bR^d_0} |D^m(\vu\zeta_{k+1})|^2\,dx\,dt
+ N \varepsilon^{\frac{l}{l-m}} \frac{2^{2mk}}{(R-r)^{2m}} \rB.
$$
By Leibniz's rule, we estimate $\rA_k$ by
$$
\rA_k \le \varepsilon \rA_{k+1} + N\int_{\bR^d_0} |\zeta_k D^m\vu|^2 \,dx\,dt
+ N \frac{2^{2mk}}{(R-r)^{2m}} \sum_{l=0}^{m-1} \varepsilon^{\frac{l}{l-m}} \rB.
$$
This combined with \eqref{eq10_08} shows
$$
\rA_k \le \varepsilon \rA_{k+1}
+N\sum_{l=0}^{m-1}\left(\varepsilon^{\frac{l}{l-m}}
+\varepsilon^{\frac{l+m}{l-m}}\right)\frac{2^{2mk}}{(R-r)^{2m}}\rB.
$$
We multiply both sides of the above inequality by $\varepsilon^k$ and sum over $k$ to obtain
$$
\sum_{k=0}^{\infty}\varepsilon^k \rA_k
\le \sum_{k=1}^{\infty}\varepsilon^k \rA_k
+ N\sum_{l=0}^{m-1}\left(\varepsilon^{\frac{l}{l-m}}+\varepsilon^{\frac{l+m}{l-m}}\right)
(R-r)^{-2m} \sum_{k=0}^{\infty}(2^{2m}\varepsilon)^k \rB.
$$
Choose $\varepsilon = 2^{-2m-1}$ and observe that $\sum_{k=0}^{\infty}\varepsilon^k\rA_k < \infty$.
Then the above inequality gives
$$
\rA_0 \le N(R-r)^{-2m}\rB,
$$
which clearly implies the desired inequality \eqref{eq2010_10} when $j=m$.

Since
$$
\rA_0 = \int_{\bR^d_0} |D^m(\vu\zeta_0)|^2 \,dx\,dt,
$$	
the proof above shows that
$$
\int_{\bR^d_0} |D^m(\vu\zeta_0)|^2 \,dx\,dt
\le N(d,m,n,\delta) (R-r)^{-2m}\int_{Q_R}|\vu|^2 \,dx\,dt.
$$
Then for $0 < j < m$, by interpolation inequalities as well as the above inequality,
\begin{multline*}
\int_{Q_r}|D^j\vu|^2\,dx\,dt
\le \int_{\bR^d_0}|D^j(\vu\zeta_0)|^2\,dx\,dt
\le (R-r)^{-2j}\int_{\bR^d_0}|\vu\zeta_0|^2\,dx\,dt\\
+ N (R-r)^{2(m-j)}\int_{\bR^d_0}|D^m(\vu\zeta_0)|^2\,dx\,dt
\le N (R-r)^{-2j}\int_{Q_R}|\vu|^2 \,dx\,dt,	
\end{multline*}
where $N=N(d,m,n,\delta)$. The lemma is proved.
\end{proof}

We are going to use the following Poincar\'e type inequality, which
generalizes a result in \cite[Sect. 3]{Krylov_2005}
\begin{lemma}
                        \label{lemPoin}
Let $p\in [1,\infty)$, $R\in (0,\infty)$, $\vu\in C^\infty_{\text{loc}}(\bR^{d+1})$.
Suppose that $\vu$ satisfies
\begin{equation}
                                        \label{eq16.11}
\vu_t+(-1)^m \cL \vu=0
\end{equation}
in $Q_R$. Let $\vP=\vP(x)$ be the vector-valued polynomial of order $m-1$ such that
$$
(D^k \vP)_{Q_R}=(D^k \vu)_{Q_R},\quad k=0,1,\ldots,m-1,
$$ and let $\vv=\vu-\vP$. Then for each $k=0,1,\ldots,m-1$, we have
\begin{equation}
                            \label{eq15.39}
\|D^k \vv\|_{L_p(Q_R)}\le NR^{m-k}\|D^m\vu\|_{L_p(Q_R)},
\end{equation}
where $N=N(d,m,n,\delta, p)>0$.
\end{lemma}
\begin{proof}
By a simple scaling, without loss of generality we may assume that $R=1$.
Take a function $\zeta\in C_0^\infty(B_1)$ with unit integral. For any $k=0,1,\ldots,m-1$
and $t\in (-1,0)$, let
$$
\vg_{k}(t)=\int_{B_1}\zeta(y)D^{k}\vv(t,y)\,dy.
$$
Then for any $t\in (-1,0)$, by H\"older's inequality and Poincar\'e's inequality,
we have
$$
\int_{B_1}|D^{k}\vv(t,x)-\vg_{k}(t)|^p\,dx
=\int_{B_1}\Big|\int_{B_1}(D^{k}\vv(t,x)-D^{k}\vv(t,y))\zeta(y)\,dy\Big|^p\,dx
$$
\begin{equation}
                            \label{eq16.05}
\le N\int_{B_1}\int_{B_1}|D^{k}\vv(t,x)-D^{k}\vv(t,y)|^p\,dy\,dx
\le N\int_{B_1}|D^{k+1}\vv(t,x)|^p\,dx.
\end{equation}
Now let $\vc_{k}=\int_{-1}^0 \vg_{k}(t)\,dt$ be a constant vector. Since
$$
\int_{Q_1}D^{k}\vv\,dx\,dt=0,
$$
by the triangle inequality, \eqref{eq16.05}, and Poincar\'e's inequality, we get
$$
\|D^{k}\vv\|_{L_p(Q_1)}\le N\|D^{k}\vv-\vc_{k}\|_{L_p(Q_1)}
\le N\|D^{k}\vv-\vg_{k}\|_{L_p(Q_1)}+N\|\vg_{k}-\vc_{k}\|_{L_p(Q_1)}
$$
\begin{equation}
                                \label{eq16.08}
\le N\|D^{k+1}\vv\|_{L_p(Q_R)}+N\|\partial_t \vg_{k}\|_{L_p((-1,0))}.
\end{equation}
By the definition of $\vg_{k}$, \eqref{eq16.11} and integration by parts,
$$
\partial_t \vg_{k}(t)=\int_{B_1}\zeta(y)D^{k}\partial_t\vv(t,y)\,dy
=\int_{B_1}\zeta(y)D^{k}\partial_t\vu(t,y)\,dy
$$
$$
=(-1)^{m+1}\int_{B_1}\zeta(y)D^{k}\cL \vu(t,y)\,dy
$$
$$
=(-1)^{k+1}\int_{B_1}(D^{k}D^\alpha\zeta)(y)a_{\alpha\beta}D_\beta \vu(t,y)\,dy.
$$
Thus, by H\"older's inequality, for any $t\in (-1,0)$,
\begin{equation}
                                    \label{eq16.19}
|\partial_t \vg_{k}(t)|\le N\|D^m\vu(t,\cdot)\|_{L_p(B_1)}.
\end{equation}
Notice that $D^m\vv=D^m \vu$. Combining \eqref{eq16.08} and \eqref{eq16.19} yields \eqref{eq15.39} by an induction on $k$.
\end{proof}

Now we prove an estimate for $D^{\alpha}\vu$ when $\alpha=(\alpha_1,\ldots,\alpha_d)$ satisfies $|\alpha| \ge m$ and $\alpha_1 \le m$.

\begin{corollary}
                                    \label{cor3}
Let $0<r<R<\infty$ and $a_{\alpha\beta}=a_{\alpha\beta}(x_1)$, $|\alpha|=|\beta|=m$. Assume that $\vu\in C_{\text{loc}}^\infty(\bR^{d+1})$ satisfies
\eqref{eq2010_02} in $Q_R$. Then for any multi-index $\alpha=(\alpha_1,\ldots,\alpha_d)$ such that $|\alpha|\ge m$ and $\alpha_1\le m$,
we have
\begin{equation*}
\|D^\alpha\vu\|_{L_2(Q_r)}
\le N\|D^m\vu\|_{L_2(Q_R)},
\end{equation*}
where $N=N(d,m,n,\delta, R, r, \alpha)$.
\end{corollary}

\begin{proof}
Since $D_{x'}^{\alpha'} \vu$ also satisfies \eqref{eq2010_02}, by applying Lemma \ref{lem2010_03} repeatedly, we obtain
\begin{equation}
                                                \label{eq11.23}
\|D^\alpha\vu\|_{L_2(Q_r)}\le N\sum_{k<m}\|D^k\vu\|_{L_2(Q_R)}
\end{equation}
for any $\alpha$ with $\alpha_1\le m$.
Now let $\vP=\vP(x)$ be the vector-valued polynomial of order $m-1$ such that
$$
\left(D^k \vP\right)_{Q_R} = \left(D^k \vu\right)_{Q_R},\quad k=0,1,\ldots,m-1,
$$ and set $\vv=\vu-\vP$. Then we have
$$
\left(D^k \vv\right)_{Q_R}=0,\quad k=0,1,\ldots,m-1.
$$
Since $\vv$ satisfies \eqref{eq2010_02} in $Q_R$,
by \eqref{eq11.23} with $\vv$ in place of $\vu$ we obtain
$$
\|D^\alpha\vu\|_{L_2(Q_r)}=\|D^\alpha\vv\|_{L_2(Q_r)}\le N\sum_{k<m}\|D^k\vv\|_{L_2(Q_R)}
$$
$$
\le N \|D^m\vv\|_{L_2(Q_R)}=N\|D^m\vu\|_{L_2(Q_R)}
$$
for any $\alpha$ satisfying $|\alpha|\ge m$ and $\alpha_1\le m$,
where the second inequality is due to Lemma \ref{lemPoin}.
\end{proof}

\begin{lemma}
                                    \label{lem2010_1}
Let $0<r<R<\infty$ and $a_{\alpha\beta}=a_{\alpha\beta}(x_1)$, $|\alpha|=|\beta|=m$. Assume that $\vu\in C_{\text{loc}}^\infty(\bR^{d+1})$ satisfies
\eqref{eq2010_02} in $Q_R$. Then for any integers $i \ge 1$, $j \ge 0$, and any multi-index $\alpha=(\alpha_1,\ldots,\alpha_d)$ such that $|\alpha|\ge m$ and $\alpha_1\le m$,
we have
\begin{equation}
                                               \label{eq2010_08}
\|\partial_t^i\vu\|_{L_2(Q_r)}
+\|\partial_t^jD^\alpha\vu\|_{L_2(Q_r)}
\le N\|D^m\vu\|_{L_2(Q_R)},
\end{equation}
where $N=N(d,m,n,\delta, R, r, \alpha, i,j)$.
\end{lemma}

\begin{proof}
Since $\vu_t$ satisfies \eqref{eq2010_02} in $Q_R$, by Lemma \ref{lem2010_03}
$$
\|D^m \vu_t\|_{L_2(Q_{r_1})} \le N\|\vu_t\|_{L_2(Q_{r_2})},
\quad r \le r_1 < r_2 \le R.
$$
Using this inequality, Corollary \ref{cor3}, and the fact that $\partial_t^i D_{x'}^{\alpha'} \vu$ also satisfies \eqref{eq2010_02} in $Q_R$,
we see that, in order to prove \eqref{eq2010_08}, it is enough to show
\begin{equation}
								\label{eq10_13}
\|\vu_t\|_{L_2(Q_r)}\le N(d,m,n,\delta,R,r)\|D^m\vu\|_{L_2(Q_R)},	
\end{equation}
where $R$ may be a smaller one than that in \eqref{eq2010_08}.
Take $r_k$, $s_k$, and $\zeta_k$ from the proof of Lemma \ref{lem2010_03}.
Set
$$
Q^{(k)} = Q_{r_k},\quad
\tilde Q^{(k)} = Q_{s_k},
\quad
k = 0, 1, 2, \ldots.
$$
Also set
$$
\rA_k = \| \vu_t\|_{L_2(Q^{(k)})}^2,
\quad
\rB = \| D^m\vu \|_{L_2(Q_R)}^2.
$$
We now apply $\vu_t\zeta_k^2$ to \eqref{eq2010_02} as a test function to get
\begin{equation}
								\label{eq10_12}
\int_{Q_R}|\vu_t \zeta_k|^2\,dx\,dt
= -\int_{Q_R}(D^{\alpha}(\vu_t \zeta_k^2), a_{\alpha\beta}D^{\beta}\vu)\,dx\,dt.	
\end{equation}
To estimate the terms in the right hand side,
we first observe that,
thanks to the fact that $\vu_t$ also satisfies \eqref{eq2010_02} in $Q_R$,
by applying \eqref{eq2010_10} with $s_k$ and $r_{k+1}$ in place
of $r$ and $R$, respectively,
$$
\int_{\tilde Q^{(k)}}|D^l \vu_t|^2\,dx\,dt
\le N (r_{k+1}-s_k)^{-2l}\int_{Q^{(k+1)}}|\vu_t|^2\,dx\,dt
$$
$$
= N\frac{2^{2lk}}{(R-r)^{2l}}\int_{Q^{(k+1)}}|\vu_t|^2\,dx\,dt,
\quad
l=0,1,\ldots,m.
$$
Hence, for each $l=0,1,\ldots,m$, by Young's inequality
$$
\int_{Q_R} |D^{m-l}(\zeta_k^2)| |D^l \vu_t| |D^m \vu| \,dx\,dt
$$
$$
\le \frac{2^{(m-l)k}}{(R-r)^{(m-l)}}
\left(\varepsilon_0\int_{\tilde Q^{(k)}}|D^l \vu_t|^2\,dx\,dt
+\frac{1}{4\varepsilon_0}\int_{Q_R}|D^m\vu|^2\,dx\,dt\right)
$$
$$
\le \varepsilon\int_{Q^{(k+1)}}|\vu_t|^2\,dx\,dt
+N\frac{1}{\varepsilon}\frac{2^{2mk}}{(R-r)^{2m}}\int_{Q_R}|D^m\vu|^2\,dx\,dt,
$$
where $\varepsilon>0$ is an arbitrary real number.
This along with \eqref{eq10_12} and Leibniz's rule implies that
$$
\rA_k \le \varepsilon \rA_{k+1} + N\frac{1}{\varepsilon}\frac{2^{2mk}}{(R-r)^{2m}}\rB.
$$
Finally, we prove \eqref{eq10_13} following the same argument as in the proof of Lemma \ref{lem2010_03}.
\end{proof}

\subsection{Maximal and sharp functions}
We recall the maximal function theorem and the Fefferman-Stein theorem.
Let
$$
\cQ=\{Q_r(X): X=(t,x) \in \bR^{d+1}, r \in (0,\infty)\}.
$$
For a function $g$ defined in $\bR^{d+1}$,
the (parabolic) maximal and sharp function of $g$ are given by
$$
\cM  g (t,x) = \sup_{Q \in \cQ, (t,x) \in Q} \dashint_{Q} |g(s,y)| \, dy \,ds,
$$
$$
g^{\#}(t,x) = \sup_{Q \in \cQ, (t,x) \in Q} \dashint_{Q} |g(s,y) -
(g)_{Q}| \, dy\,ds.
$$
It is well known that
$$
\| g \|_{L_p(\bR^{d+1})} \le N \| g^{\#} \|_{L_p(\bR^{d+1})},
\quad
\| \cM  g \|_{L_p(\bR^{d+1})} \le N \| g\|_{L_p(\bR^{d+1})},
$$
if $g \in L_p(\bR^{d+1})$, where $1 < p < \infty$ and $N = N(d,p)$.
Indeed, the first of the inequalities above is due to the Fefferman-Stein theorem on sharp functions
and the second one to the Hardy-Littlewood maximal function theorem (this inequality also holds trivially when $p = \infty$).

Theorem \ref{th081201} below is from \cite{Krylov08}
and can be considered as a generalized version of the Fefferman-Stein Theorem.
To state the theorem,
let
$$
\bC_l = \{ C_l(i_0, i_1, \ldots, i_d), i_0, i_1, \ldots, i_d \in \bZ \},
\quad l \in \bZ
$$
be the collection of
partitions given by parabolic dyadic cubes in $\bR^{d+1}$
\begin{multline*}
[ i_0 2^{-2ml}, (i_0+1)2^{-2ml} ) \times [ i_1 2^{-l}, (i_1+1)2^{-l} ) \times \ldots \times [ i_d 2^{-l}, (i_d+1)2^{-l} ).
\end{multline*}

\begin{theorem}							\label{th081201}
Let $p \in (1, \infty)$, and $U,V,F\in L_{1,\text{loc}}(\bR^{d+1})$.
Assume that we have $|U| \le V$
and, for each $l \in \bZ$ and $C \in \bC_l$,
there exists a measurable function $U^C$ on $C$
such that $|U| \le U^C \le V$ on $C$ and
\begin{equation*}							 
\int_C |U^C - \left(U^C\right)_C| \,dx\,dt
\le \int_C F(t,x) \,dx\,dt.
\end{equation*}
Then
$$
\| U \|_{L_p(\bR^{d+1})}^p
\le N(d,p) \|F\|_{L_p(\bR^{d+1})}\| V \|_{L_p(\bR^{d+1})}^{p-1},
$$
provided that $F,V\in L_p(\bR^{d+1})$.
\end{theorem}

\mysection{Interior H\"older estimates}
                        \label{sec3.2}

By using the $L_2$ estimates obtained in Section \ref{sec3.1}, in this section we shall
derive interior H\"older estimates of derivatives of $\vu$.
As usual, for $\mu\in (0,1)$ and a function $u$ defined on $\cD \subset \bR^{d+1} $, we denote
$$
[u]_{C^{\mu}(\cD)}
= \sup_{\substack{(t,x),(s,y)\in\cD\\(t,x)\ne(s,y)}}\frac{|u(t,x)-u(s,y)|}{|t-s|^{\mu/2}+|x-y|^{\mu}},\quad
\|u\|_{C^{\mu}(\cD)}=[u]_{C^{\mu}(\cD)}+\|u\|_{L_{\infty}(\cD)}.
$$

\begin{lemma}
 				\label{lemma100}
Let $a_{\alpha\beta}=a_{\alpha\beta}(x_1)$.
Assume that $\vu\in C_{\text{loc}}^\infty(\bR^{d+1})$ satisfies
\eqref{eq2010_02} in $Q_2$.
Then for any $\alpha=(\alpha_1,\ldots,\alpha_d)$ satisfying $|\alpha|=m$ and $\alpha_1<m$ we have
\begin{equation*}	
\left\|D^\alpha\vu\right\|_{C^{1/2}(Q_1)}
\le N \|D^m\vu\|_{L_2(Q_2)},
\end{equation*}
where $N=N(d,m,n,\delta)>0$.
\end{lemma}
\begin{proof}
Thanks to the well-known interpolation inequality, it is sufficient
to estimate $[D^\alpha u]_{C^{1/2}(Q_1)}$.
The proof is based on a convenient form of the Sobolev inequality.
By the triangle inequality, we have
$$
\sup_{\substack{(t,x),(s,y) \in Q_1\\(t,x) \ne (s,y)}}\frac{|D^\alpha \vu(t,x) - D^\alpha \vu(s,y)|}{|t-s|^{1/4}+|x-y|^{1/2}}\le I+J,
$$
where
$$
I:=
\sup_{\substack{x_1, y_1\in(-1,1), x_1 \ne y_1\\(t,x') \in Q_1'}}\frac{|D^\alpha \vu(t,x_1,x') - D^\alpha \vu(t,y_1,x')|}{|x_1-y_1|^{1/2}},
$$
$$
J:=\sup_{\substack{y_1\in(0,1)\\(t,x'),(s,y') \in Q_1',(t,x')\ne (s,y')}}
\frac{|D^\alpha \vu(t,y_1,x') - D^\alpha \vu(s,y_1,y')|}{|t-s|^{1/4}+|x'-y'|^{1/2}}.
$$

{\em Estimate of $I$:} By the Sobolev embedding theorem $D^\alpha \vu(t,x_1,x')$, as a function of $x_1\in (-1,1)$, satisfies
\begin{equation}							\label{eq01}
\sup_{\substack{x_1, y_1\in(-1,1)\\x_1 \ne y_1}}
\frac{|D^\alpha \vu(t,x_1,x')-D^\alpha \vu(t,y_1,x')|}{|x_1-y_1|^{1/2}}
\le N \| D^\alpha \vu(t,\cdot,x') \|_{W_2^1(-1,1)}.
\end{equation}
On the other hand, there exists a positive integer $k$
such that $D^\alpha \vu(t,x_1,x')$ and $D_1D^\alpha \vu(t,x_1,x')$, as functions of $(t,x') \in Q_1'$,
satisfy
$$
\sup_{(t,x') \in Q_1'}\left(|D^\alpha \vu(t,x_1,x')|+|D_1D^\alpha \vu (t,x_1,x')|\right)
$$
$$
\le N\|D^\alpha\vu(\cdot,x_1,\cdot)\|_{W_2^k(Q_1')}+
N\|D_1D^\alpha\vu(\cdot,x_1,\cdot)\|_{W_2^k(Q_1')}.
$$
The inequality above implies that, for all $(t,x') \in Q_1'$,
$$
\int_{-1}^1|D^\alpha\vu(t,x_1,x')|^2\, dx_1
+ \int_{-1}^1|D_1D^\alpha\vu(t,x_1,x')|^2\,dx_1
$$
$$
\le N \sum_{\substack{i+|\beta|\le m+1+k\\\beta_1\le m}} \| \partial_t^i D^\beta \vu\|^2_{L_2(Q_{\sqrt 2})}.
$$
This combined with \eqref{eq01} shows that
$$
I
\le N \sum_{\substack{i+|\beta|\le m+1+k\\\beta_1\le m}} \| \partial_t^i D^\beta \vu\|_{L_2(Q_{\sqrt 2})}
\le N \| D^m \vu \|_{L_2(Q_2)},
$$
where the last inequality is due to Lemma \ref{lem2010_1}.

{\em Estimate of $J$:} By the Sobolev embedding theorem,
we find a positive integer $k$ such that
$D^\alpha\vu(t,y_1,x')$, as a function of $(t,x') \in Q_1'$,
satisfies
\begin{equation}							\label{eq02}
\sup_{\substack{(t,x'), (s,y') \in Q_1'\\(t,x')\ne (s,y')}}
\frac{|D^\alpha\vu(t,y_1,x')-D^\alpha\vu(s,y_1,y')|}{|t-s|^{1/4}+|x'-y'|^{1/2}}
\le N \| D^\alpha\vu(\cdot,y_1, \cdot) \|_{W_2^k(Q_1')}.
\end{equation}
For each $i,j$ such that $i+j \le k$,
$\partial_t^iD^j_{x'}D^\alpha\vu(t,y_1,x')$, as a function of $y_1 \in (-1,1)$,
satisfies
$$
\sup_{y_1\in(-1,1)}|\partial_t^iD^j_{x'}D^\alpha\vu(t,y_1,x')|
$$
$$
\le N\|\partial_t^iD^j_{x'}D^\alpha\vu(t,\cdot,x')\|_{L_2(-1,1)}
+N \|\partial_t^iD^j_{x'}D_1D^\alpha\vu(t,\cdot,x')\|_{L_2(-1,1)}.
$$
This together with \eqref{eq02} and Lemma \ref{lem2010_1} gives
$$
J
\le N\sum_{\substack{i+|\beta|\le m+1+k\\\beta_1\le m}} \| \partial_t^iD^\beta \vu\|_{L_2(Q_{\sqrt 2})}
\le N \|D^m\vu\|_{L_2(Q_2)}.
$$
This completes the proof of the lemma.
\end{proof}

For $\lambda\ge0$, let
\begin{equation}
								\label{eq10_100}
U=\sum_{|\alpha| \le m}\lambda^{\frac 1 2-\frac {|\alpha|}{2m}} |D^\alpha \vu|,
\quad
U'=\sum_{|\alpha|\le m,\alpha_1<m} \lambda^{\frac 1 2-\frac {|\alpha|} {2m}} |D^\alpha \vu|.	 
\end{equation}
\begin{corollary}
 				\label{cor3.5}
Let $a_{\alpha\beta}=a_{\alpha\beta}(x_1)$ and $\lambda\ge0$.
Assume that $\vu\in C_{\text{loc}}^\infty(\bR^{d+1})$ satisfies
\begin{equation}
 					\label{eq10.32}
\vu_t+(-1)^m\cL_0 \vu+\lambda \vu=0
\end{equation}
in $Q_2$.
Then we have
\begin{equation*}	
\left\|U'\right\|_{C^{1/2}(Q_1)}
\le N \|U\|_{L_2(Q_2)},
\end{equation*}
where $N=N(d,m,n,\delta)>0$.
\end{corollary}
\begin{proof}
The case when $\lambda=0$ follows from Lemma \ref{lemma100}. To deal with the case $\lambda>0$, we follow an idea by S. Agmon. Let $
\eta(y)=\cos(\lambda^{1/(2m)}y)+\sin(\lambda^{1/(2m)}y)
$
so that $\eta$ satisfies
$$
D^{2m}\eta=(-1)^m\lambda \eta,\quad
\eta(0)=1,\quad |D^j\eta(0)|=\lambda^{j/(2m)}\,\,\,j=1,2,\ldots.
$$
Let $z = (x,y)$ be a point in $\bR^{d+1}$, where $x \in \bR^{d}$, $y \in \bR$,
and $\hat{\vu}(t,z)$ and $\hat{Q}_r$ be given by
$$
\hat{\vu}(t,z) = \hat{\vu}(t,x,y) = \vu(t,x)\eta(y),
\quad
\hat{Q}_r = (-r^{2m},0)\times\,\{ |z| < r: z \in \bR^{d+1} \}.
$$
Since $\hat{\vu}$ satisfies, in $\hat{Q}_2$,
$$
\hat{\vu}_t+(-1)^m\cL_0\hat{\vu}+(-1)^mD^{2m}_y(\hat{\vu}) = 0,
$$
by Lemma \ref{lemma100} applied to $\hat{\vu}$ we have
\begin{equation}								 \label{eq0804}
\left\| D_z^\beta\hat{\vu}\right\|_{C^{1/2}(\hat{Q}_1)}
\le N(d,m,n,\delta) \|D^m_z\hat{\vu}\|_{L_2(\hat{Q}_2)}
\end{equation}
for any $\beta=(\beta_1,\ldots,\beta_{d+1})$ satisfying $|\beta|=m$ and $\beta_1<m$.
Notice that for any $\alpha=(\alpha_1,\ldots,\alpha_d)$ satisfying $|\alpha|\le m$ and $\alpha_1<m$,
$$
\lambda^{\frac 1 2-\frac {|\alpha|} {2m}} \left\|D^\alpha\vu\right\|_{C^{1/2}(Q_1)}
\le N\left\| D^\beta_z\hat{\vu} \right\|_{C^{1/2}(\hat{Q}_1)},\quad \beta=(\alpha_1,\ldots,\alpha_d,m-|\alpha|)
$$
and $D_z^m \hat{\vu}$ is a linear combination of
$$
\lambda^{\frac 1 2-\frac k {2m}}\cos( \lambda^{\frac 1 {2m}} y) D^k_x\vu,
\quad
\lambda^{\frac 1 2-\frac k {2m}}\sin( \lambda^{\frac 1 {2m}} y) D^k_x\vu,\quad k=0,1,\ldots,m.
$$
Thus  the right-hand side of \eqref{eq0804} is less than the right-hand side of the inequality in the lemma.
The lemma is proved.
\end{proof}

Let $\bar\alpha=me_1=(m,0,\ldots,0)$. In the remaining part of this section, we shall establish a H\"older estimate of
\begin{equation}
								\label{eq10_101}
\Theta:=\sum_{|\beta|=m}a_{\bar\alpha\beta}D^\beta \vu.	
\end{equation}

We make use of the following elementary lemma.
\begin{lemma}
                                \label{lemA.1}
Let $r\in (0,\infty)$, $k\ge 1$ be an integer, $p\in [1,\infty]$, and $u\in L_p([0,r])$. Assume that $D^k u=f_0+Df_1+\ldots+D^{k-1}f_{k-1}$ in $(0,r)$, where $f_j\in L_1([0,r]),j=0,\ldots,k-2$ and $f_{k-1}\in L_p([0,r])$. Then
$Du \in L_p([0,r])$ and
\begin{equation}
                                        \label{eq11.06}
\|Du\|_{L_p([0,r])}\le N\|u\|_{L_1([0,r])}+N\|f_{k-1}\|_{L_p([0,r])}+N\sum_{j=0}^{k-2}
\|f_j\|_{L_1([0,r])},
\end{equation}
where $N=N(k, r)>0$.
\end{lemma}

\begin{proof}
Thanks to scaling, we may assume $r=1$.
For any function $f\in L_1([0,1])$, we define its anti-derivative $\mathcal{I}f\,:\,[0,1]\to \bR$ as $\mathcal{I}f(x)=\int_0^x f(t)\,dt$. It is easily seen that $D^k u=D^{k-1}\tilde f_{k-1}$, where
$$
\tilde f_{k-1}=f_{k-1}+\cI f_{k-2}+\cI^2 f_{k-3}+\ldots+\cI^{k-1} f_{0},$$
and
$$
\|\tilde f_{k-1}\|_{L_p([0,1])}\le N\|f_{k-1}\|_{L_p([0,1])}+N\sum_{j=0}^{k-2}\|f_j\|_{L_1([0,1])}.
$$
Therefore, without loss of generality we may assume $f_j=0$ for $j=0,\ldots,k-2$. Under this assumption, we have for some constant $c_j,j=0,1,\ldots,k-1$,
\begin{equation}
                                    \label{eq12.06}
u(x)=(\cI f_{k-1})(x)+c_0+2c_1x+\ldots+kc_{k-1}x^{k-1}.
\end{equation}
We claim that
\begin{equation}
                                    \label{eq11.30}
|c_j|\le N\|u\|_{L_1([0,1])}+N\|f_{k-1}\|_{L_1([0,1])},\quad j=0,1,\ldots,k-1,
\end{equation}
which immediately yields \eqref{eq11.06}. To prove the claim, we integrate both sides of \eqref{eq12.06} on $[0,j/k],j=1,2,\ldots,k$ to get
$$
\int_0^{j/k}u(x)\,dx=\int_0^{j/k}(\cI f_{k-1})(x)\,dx+c_0(j/k)+c_1(j/k)^2+\ldots+c_{k-1}(j/k)^k.
$$
The claim \eqref{eq11.30} then follows since the matrix $[(j/k)^{i}]_{i,j=1}^k$ is nondegenerate and
\begin{multline*}
\Big|\int_0^{j/k}(u(x)-(\cI f_{k-1})(x))\,dx\Big|\\
\le N\|u\|_{L_1([0,1])}+N\|f_{k-1}\|_{L_1([0,1])},\quad j=1,2,\ldots,k.
\end{multline*}

\end{proof}

\begin{corollary}
                            \label{corA.2}
Let $k\ge 1$ be an integer, $r\in (0,\infty)$, $p\in [1,\infty]$, $\cD=[0,r]^{d+1}$, and $u(t,x)\in L_p(\cD)$. Assume that $D_1^k u=f_0+D_1f_1+\ldots+D_1^{k-1}f_{k-1}$ in $\cD$, where $f_j\in L_p(\cD),j=0,\ldots,k-1$.
Then  $D_1 u \in L_p(\cD)$ and
\begin{equation*}
\|D_1 u\|_{L_p(\cD)}\le N\|u\|_{L_p(\cD)}+N\sum_{j=0}^{k-1}\|f_j\|_{L_p(\cD)},
\end{equation*}
where $N=N(d,k, r)>0$.
\end{corollary}
\begin{proof}
The corollary follows from Lemma \ref{lemA.1} by first fixing $(t,x')$ and then integrating with respect to $(t,x')$.
\end{proof}

\begin{lemma}
 				\label{lemma01}
Let $0<r<R<\infty$ and $a_{\alpha\beta}=a_{\alpha\beta}(x_1)$.
Assume $\vu\in C_{\text{loc}}^\infty(\bR^{d+1})$ satisfies
\eqref{eq2010_02}
in $Q_R$.
Then, for any nonnegative integers $i,j$,
\begin{equation}	
 			\label{eq03}
\|\partial_t^iD^j_{x'} \Theta\|_{L_2(Q_r)}
+ \|\partial_t^i D^j_{x'} D_1 \Theta \|_{L_2(Q_r)}
\le N \|D^m\vu\|_{L_2(Q_R)},
\end{equation}
where $N=N(d,m,n,r,R,\delta, i, j)>0$.
\end{lemma}
\begin{proof}
Obviously, we have
\begin{equation}
                                    \label{eq11.07}
\|\Theta \|_{L_2(Q_r)}
\le N \|D^m\vu\|_{L_2(Q_r)}.
\end{equation}
Thus as noted at the beginning of the proof of Lemma \ref{lem2010_1}, it suffices to prove
\begin{equation}
                        \label{eq11.55}
\|D_1 \Theta \|_{L_2(Q_r)}
\le N \|D^m\vu\|_{L_2(Q_R')},
\end{equation}
where $R'=(r+R)/2 $.
From \eqref{eq2010_02}, in $Q_R$ we have
$$
D_1^m\Theta=(-1)^{m+1}\vu_t-\sum_{\substack{|\alpha|=|\beta|=m\\\alpha_1<m}}D_\alpha(a_{\alpha\beta}D_\beta \vu)
$$
$$
=(-1)^{m+1}\vu_t-\sum_{\substack{|\alpha|=|\beta|=m\\\alpha_1<m}}D_1^{\alpha_1}
(a_{\alpha\beta}D_{x'}^{\alpha'}D^\beta \vu),
$$
where $\alpha = (\alpha_1,\alpha_2,\ldots,\alpha_d) = (\alpha_1,\alpha')$.
Then the estimate \eqref{eq11.55} follows from Corollary \ref{corA.2} with a covering argument and Lemma \ref{lem2010_1}. The lemma is proved.
\end{proof}

The following H\"older estimate is deduced from Lemma \ref{lemma01} in the same way as Lemma \ref{lemma100} and Corollary \ref{cor3.5} are deduced from Lemma \ref{lem2010_1}.

\begin{lemma}
 				\label{lemma3.6}
Let $a_{\alpha\beta}=a_{\alpha\beta}(x_1)$ and $\lambda\ge0$.
Assume that $\vu\in C_{\text{loc}}^\infty(\bR^{d+1})$ satisfies \eqref{eq10.32} in $Q_2$.
Then we have
\begin{equation*}	
\left\|\Theta\right\|_{C^{1/2}(Q_1)}
\le N \|U\|_{L_2(Q_2)},
\end{equation*}
where $N=N(d,m,n,\delta)>0$.
\end{lemma}

Since the matrix $[a_{\bar\alpha\bar\alpha}^{ij}]_{i,j=1}^n$ is positive definite, we obtain the following estimate by using Corollary \ref{cor3.5} and Lemma \ref{lemma3.6}.

\begin{corollary}
                                        \label{cor3.10}
Let $a_{\alpha\beta}=a_{\alpha\beta}(x_1)$ and $\lambda\ge0$.
Assume that $\vu\in C_{\text{loc}}^\infty(\bR^{d+1})$ satisfies \eqref{eq10.32} in $Q_2$.
Then we have
\begin{equation*}	
\|U\|_{L_\infty(Q_1)}
\le N \|U\|_{L_2(Q_2)},
\end{equation*}
where $N=N(d,m,n,\delta)>0$.
\end{corollary}

\mysection{Estimates of mean oscillations}
                    \label{sec4}

Recall the definitions of $U$, $U'$, and $\Theta$ in \eqref{eq10_100}
and \eqref{eq10_101}, respectively. With the preparations in the previous section, we obtain the following estimates of mean oscillations of $U'$ and $\Theta$.

\begin{lemma}
 				\label{lem4.1}
Let $r\in (0,\infty)$, $\kappa\in [2,\infty)$, $\lambda\ge 0$, and $a_{\alpha\beta}=a_{\alpha\beta}(x_1)$.
Assume $\vu\in C_{\text{loc}}^\infty(\bR^{d+1})$ satisfies
\begin{equation*}
\vu_t+(-1)^m\cL_0 \vu + \lambda \vu=0
\end{equation*}
in $Q_{\kappa r}$.
Then we have
\begin{equation}	
 			\label{eq5.111}
\left(|U'-(U')_{Q_r}|\right)_{Q_r}
+\left(|\Theta-(\Theta)_{Q_r}|\right)_{Q_r}
\le N\kappa^{-1/2} (U^2)_{Q_{\kappa r}}^{1/2},
\end{equation}
where $N=N(d,m,n,\delta)>0$.
\end{lemma}
\begin{proof}
By a scaling argument, we may assume $r=2/\kappa$.
Then by Lemma \ref{lemma3.6},
we have
$$
\left(|\Theta-(\Theta)_{Q_r}|\right)_{Q_r}
\le Nr^{1/2}[\Theta]_{C^{1/2}(Q_1)}
\le N \kappa^{-1/2} (U^2)_{Q_{2}}^{1/2}.
$$
The first term on the left-hand side of \eqref{eq5.111} is estimated similarly by using Corollary \ref{cor3.5}.
\end{proof}

For $\vf_\alpha= (f_\alpha^1, \ldots, f_\alpha^n)^{\text{tr}}$, we denote
$$
F=\sum_{|\alpha|\le m}\lambda^{\frac {|\alpha|} {2m}-\frac 1 2}|\vf_\alpha|.
$$
\begin{lemma}
 				\label{lem4.2}
Let $r\in (0,\infty)$, $\kappa\in [4,\infty)$, $\lambda> 0$, $\vf_\alpha \in L_{2,\text{loc}}(\bR^{d+1})$, $|\alpha|\le m$, and $a_{\alpha\beta}=a_{\alpha\beta}(x_1)$.
Assume $\vu\in C_{\text{loc}}^\infty(\bR^{d+1})$ satisfies
\begin{equation*}
\vu_t+(-1)^m\cL_0 \vu + \lambda \vu=\sum_{|\alpha|\le m}D^\alpha \vf_\alpha
\end{equation*}
in $Q_{\kappa r}$.
Then we have
\begin{equation}	
 			\label{eq5.11}
\left(|U'-(U')_{Q_r}|\right)_{Q_r}
+\left(|\Theta-(\Theta)_{Q_r}|\right)_{Q_r}
\le N\kappa^{-1/2}(U^2)_{Q_{\kappa r}}^{1/2}+N\kappa^{m+\frac d2}(F^2)_{Q_{\kappa r}}^{1/2},
\end{equation}
\begin{equation}
                            \label{eq12.01}
N^{-1}U\le U'+\Theta\le NU,
\end{equation}
where $N=N(d,m,n,\delta)>0$. 
\end{lemma}

\begin{proof}
%
%
The inequality \eqref{eq12.01} follows from the definitions of $U$, $U'$ and $\Theta$ as well
as the fact that $[a_{\bar\alpha\bar\alpha}^{ij}]_{i,j=1}^n$ is positive definite.
To prove \eqref{eq5.11}, we adapt the idea in the proof of Theorem 7.1 in \cite{Krylov_2007_mixed_VMO} and take into account the presence of $\lambda$.
We can certainly assume that $\vu$ and $\vf_\alpha$ have compact supports.
In addition, we assume that $a_{\alpha\beta}$ and $\vf_\alpha$ are infinitely differentiable. If not, we take the standard mollifications and prove the estimate for the mollifications. Then we can pass to the limit because the constant $N$ in the estimate \eqref{eq5.11} is independent of the regularity of $a_{\alpha\beta}$ and $\vf_\alpha$.

Take a $\zeta \in C_0^{\infty}(\bR^{d+1})$ such that
$$
\zeta = 1
\quad
\text{on}
\quad Q_{\kappa r/2},
\quad
\zeta = 0
\quad
\text{outside}
\quad
(-(\kappa r)^{2m}, (\kappa r)^{2m})\times B_{\kappa r}.
$$

By Theorem \ref{theorem08061901}, for $\lambda > 0$, there exists a unique solution
$\cH_2^m(\bR^{d+1})$ to the equation
$$
\vw_t+(-1)^m\cL_0\vw + \lambda \vw = \sum_{|\alpha|\le m}D^\alpha( \zeta \vf_\alpha ).
$$
Since all functions and coefficients involved are infinitely differentiable,
by the classical parabolic theory, $\vw$ is infinitely differentiable.
The function $\vv := \vu - \vw$ is also infinitely differentiable and satisfies
$$
\vv_t+(-1)^m\cL_0\vv+\lambda \vv  = 0\quad \text{in} \quad Q_{\kappa r/2}.
$$
We define $V$, $V'$, $W$, and $W'$ in the same way as $U$ and $U'$.
Thus by Lemma \ref{lem4.1} (note that $\kappa /2 \ge 2$)
\begin{equation}	
 			\label{eq1003}
\left(|V'-(V')_{Q_r}|\right)_{Q_r}
+\left(|\hat\Theta-(\hat\Theta)_{Q_r}|\right)_{Q_r}
\le N\kappa^{-1/2}(V^2)_{Q_{\kappa r/2}}^{1/2},
\end{equation}
where $\hat\Theta$ is defined in the same way as $\Theta$ with $\vu$ replaced by $\vv$, i.e.
$$
\hat\Theta:=\sum_{|\beta|=m}a_{\bar\alpha\beta}D^\beta \vv,\quad \bar\alpha=me_1.
$$

Next we estimate $\vw$. By Theorem \ref{theorem08061901} we have
$$
\sum_{|\alpha|\le m}\lambda^{1-\frac {|\alpha|} {2m}} \|D^\alpha \vw \|_{L_2(\bR^d_0)}
\le N \sum_{|\alpha|\le m}\lambda^{\frac {|\alpha|} {2m}} \| \zeta\vf_\alpha \|_{L_2(\bR^d_0)}.
$$
In particular,
\begin{equation}							\label{eq1004}
\left(W^2\right)_{Q_r}^{1/2}
\le N \kappa^{m+\frac d2} (F^2)_{Q_{\kappa r}}^{1/2},\quad
\left(W^2\right)_{Q_{\kappa r}}^{1/2}
\le N (F^2)_{Q_{\kappa r}}^{1/2}.
\end{equation}

Now we are ready to prove \eqref{eq5.11}.
From \eqref{eq1003} and \eqref{eq1004}, and the fact that $\vu = \vw + \vv$,
we bound the left-hand side of \eqref{eq5.11} by
$$
\big(|V'-(V')_{Q_r}|\big)_{Q_r}
+\big(|\hat\Theta-(\hat\Theta)_{Q_r}|\big)_{Q_r}
+N\big(W^2\big)_{Q_r}^{1/2}
$$
$$
\le N\kappa^{-1/2}\big(V^2\big)_{Q_{\kappa r/2}}^{1/2}
+N \kappa^{m+\frac d2} \big(F^2\big)_{Q_{\kappa r}}^{1/2},
$$
which is less than the right-hand side of \eqref{eq5.11}.
\end{proof}


Recall that the ellipticity condition \eqref{eq11.28} is invariant under any orthogonal transformation of the coordinates.

\begin{corollary}
								\label{cor10_01}
Let $r\in (0,\infty)$, $\kappa\in [4,\infty)$, $\lambda> 0$, $\vf_\alpha \in L_{2,\text{loc}}(\bR^{d+1})$, $|\alpha|\le m$, and $a_{\alpha\beta}=a_{\alpha\beta}(y_1)$,
where $\rho$ is a $d \times d$ orthogonal matrix and $y = \rho x$.
Assume $\vu\in C_{\text{loc}}^\infty(\bR^{d+1})$ satisfies
\begin{equation}
 					\label{eq10_20}
\vu_t+(-1)^m\cL_0 \vu + \lambda \vu=\sum_{|\alpha|\le m}D^\alpha \vf_\alpha
\end{equation}
in $Q:=Q_{\kappa r}$.
Then there exist a function $U^{Q}$ depending on $Q$, and a constant $N=N(d,m,n,\delta)>0$ such that
\begin{equation}
                        \label{eq12.20}
N^{-1}U\le U^Q\le NU,
\end{equation}
\begin{equation*}	
\left(|U^Q -(U^Q)_{Q_r}|\right)_{Q_r}
\le N\kappa^{-1/2}(U^2)_{Q_{\kappa r}}^{1/2}+N\kappa^{m+\frac d2}(F^2)_{Q_{\kappa r}}^{1/2},
\end{equation*}
\end{corollary}
\begin{proof}
Since $\vu$ satisfies \eqref{eq10_20}, we see that $\vv(t,y):=\vu(t,\rho^{-1}y)$ satisfies
$$
\vv_t + (-1)^m D^{\alpha}(\tilde{a}_{\alpha\beta}(y_1)D^{\beta}\vv) + \lambda \vv = \sum_{|\alpha| \le m}D^\alpha \tilde{\vf}_{\alpha}
$$
in $Q$, where $\tilde{a}_{\bar\alpha\beta}$ are the corresponding coefficients in the $y$-coordinates and $\tilde{\vf}_{\alpha}(t,y)$, $|\alpha|=k$, is a linear combination of $\vf_{\alpha}(t,\rho^{-1}y)$, $|\alpha| = k$.

Set
$$
U^Q = V' + \tilde\Theta,
$$
where $V'$ and $\tilde\Theta$ are defined as $U'$ and $\Theta$ in \eqref{eq10_100} and \eqref{eq10_101}, but with $\vv$ and $\tilde{a}_{\alpha\beta}$ in place of $\vu$ and $a_{\alpha\beta}$, respectively.
Then since the new operator also satisfies \eqref{eq11.28}, the corollary follows from Lemma \ref{lem4.2}.
\end{proof}

\mysection{Proof of Theorem \ref{thm1}}
            \label{sec5}
In this section we finish the proof of Theorem \ref{thm1}.
First we observe that by taking a sufficiently large $\lambda_0$ and using interpolation inequalities, we can move all the lower-order terms of $\cL\vu$ to the right-hand side. Thus, in the sequel, we assume all the lower-order coefficients of $\cL$ are zero.
Recall the definitions of $\bO$ and $\bA$ above Assumption \ref{assumption20080424}, and the constant $R_0$ from Assumption \ref{assumption20080424}.

\begin{theorem}
                            \label{theorem3.42}
Let $\gamma\in (0,1)$, $\lambda>0$, and $\tau,\sigma \in (1,\infty)$, $1/\tau+1/\sigma=1$.
Assume $\vu\in C_{\text{loc}}^\infty(\bR^{d+1})$ vanishes outside $Q_{\gamma R_0}$ and
$$
\vu_t+(-1)^m\cL \vu+\lambda \vu=\sum_{|\alpha|\le m}D^\alpha \vf_\alpha
$$
in $Q_{\kappa r}(X_0)$, where $\vf_\alpha\in L_{2,\text{loc}}(\bR^{d+1})$.
Then under Assumption \ref{assumption20080424} ($\gamma$),
for each $r \in (0,\infty)$, $\kappa \ge 4$,
and $X_0:=(t_0,x_0) \in \bR^{d+1}$, there exists a function $U^{Q}$ depending on $Q:=Q_{\kappa r}(X_0)$,
such that we have \eqref{eq12.20} and
$$
\left(|U^{Q}-(U^{Q})_{Q_r(X_0)}|\right)_{Q_r(X_0)}
\le N\kappa^{-1/2}(U^2)_{Q_{\kappa r}(X_0)}^{1/2}
$$
\begin{equation}
                                                    \label{eq3.49}
+N\kappa^{m+\frac d2}\left[(F^2)_{Q_{\kappa r}(X_0)}^{1/2}+\gamma^{1/(2\sigma)}
(U^{2\tau})_{Q_{\kappa r}(X_0)}^{1/(2\tau)}\right],
\end{equation}
where $N=N(d,\delta,m,n,\tau)$.
\end{theorem}

\begin{proof}
Fix $\kappa\ge 4$, $r\in (0,\infty)$, and $X_0 \in \bR^{d+1}$.
First we consider the case when $\kappa r < R_0$.
For $Q=Q_{\kappa r}(X_0)$, from Assumption \eqref{assumption20080424} ($\gamma$),
we find $\cT_Q \in \bO$ and $\{\bar a_{\alpha\beta}\}_{|\alpha|=|\beta|=m} \in \bA$ satisfying \eqref{eq10_23}.
Then we see that $\vu$ satisfies
$$
\vu_t+(-1)^mD^{\alpha}(\bar a_{\alpha\beta}D^{\beta}\vu)
+\lambda \vu=\sum_{|\alpha|\le m}D^\alpha \hat{\vf_\alpha},
$$
where $\bar a_{\alpha\beta}= \bar a_{\alpha\beta}(y_1)$, $y = \cT_Q (x)$, and
$$
\hat{\vf_{\alpha}} = \vf_{\alpha}
+ 1_{|\alpha|=m}\sum_{|\beta|=m}(-1)^m\left(\bar a_{\alpha\beta}-a_{\alpha\beta}\right)D^{\beta}\vu.
$$
Using Corollary \ref{cor10_01} with a shift of the coordinates, there exists a function $U^Q$ satisfying \eqref{eq12.20} such that
\begin{equation}
								\label{eq10_24}
\left(|U^Q -(U^Q)_{Q_r(X_0)}|\right)_{Q_r(X_0)}
\le N\kappa^{-1/2}(U^2)_{Q_{\kappa r}(X_0)}^{1/2}+N\kappa^{m+\frac d2}(\hat F^2)_{Q_{\kappa r}(X_0)}^{1/2},	
\end{equation}
where $N = N(d,m,n,\delta)$ and $\hat F$ is defined by using $\hat{\vf_\alpha}$ in the same way as $F$.
Observe that for $|\alpha|=m$
\begin{equation}							 \label{eq081102}
\int_{Q_{\kappa r}(X_0)} |\hat{\vf_\alpha}|^2 \,dx\,dt
\le N \int_{Q_{\kappa r}(X_0)} |\vf|^2 \,dx\,dt
+ NI,
\end{equation}
where
$$
I
= \sum_{|\beta|=m}\int_{Q_{\kappa r}(X_0)}
\big| (\bar a_{\alpha\beta} - a_{\alpha\beta}) D^\beta \vu \big|^2 \,dx\,dt.
$$
By H\"{o}lder's inequality, we have
\begin{equation}							 \label{eq081103}
I \le N J_1^{1/\sigma} J_2^{1/\tau},
\end{equation}
where
$$
J_1 = \sum_{|\beta|=m}\int_{Q_{\kappa r}(X_0)} | \bar a_{\alpha\beta} - a_{\alpha\beta} |^{2\sigma} \, dx \,dt,
\quad
J_2 = \int_{Q_{\kappa r}(X_0)} |D^m \vu|^{2\tau} \, dx \,dt.
$$
Since $\kappa r < R_0$, by Assumption \ref{assumption20080424}
$$
J_1 \le N \sum_{|\beta|=m}\int_{Q_{\kappa r}(X_0)} | \bar a_{\alpha\beta} - a_{\alpha\beta} | \,dx\,dt
\le  N (\kappa r)^{d+2m} \gamma,
$$
where $N$ depends only on $d,m,n$ and $\delta$. From the above estimates for $J_1$ as well as the inequalities \eqref{eq10_24}, \eqref{eq081102},
and  \eqref{eq081103}, we conclude \eqref{eq3.49}.

In case $\kappa r \ge R_0$, we take $U^Q=U$.
By the triangle inequality and
H\"older's inequality, the left-hand side of \eqref{eq3.49} is less than
$$
2(U^2)^{1/2}_{Q_r(X_0)}\le N\kappa^{m+\frac d 2}
(U^2)^{1/2}_{Q_{\kappa r}(X_0)}
= N\kappa^{m+\frac d 2}
(1_{Q_{\gamma R_0}}U^2)^{1/2}_{Q_{\kappa r}(X_0)}
$$
$$
\le N\kappa^{m+\frac d 2}
(1_{Q_{\gamma R_0}})^{1/(2\sigma)}_{Q_{\kappa r}(X_0)}
(U^{2\tau})^{1/(2\tau)}_{Q_{\kappa r}(X_0)}
\le N\kappa^{m+\frac d 2}
\gamma^{1/(2\sigma)}(U^{2\tau})^{1/(2\tau)}_{Q_{\kappa r}(X_0)},
$$
where $N = N(d,m,n,\delta)$. The theorem is proved.
\end{proof}

\begin{corollary}							\label{cor001}
Let $\gamma \in (0,1)$, $\lambda > 0$, and $\tau, \sigma \in (1,\infty)$,
$1/\tau+1/\sigma=1$.
Assume that $\vu\in C_0^\infty(\bR^{d+1})$ vanishes outside $Q_{\gamma R_0}$ and satisfies
$$
\vu_t+(-1)^m\cL \vu+\lambda \vu=\sum_{|\alpha|\le m}D^\alpha \vf_\alpha
$$
in $\bR^{d+1}$,
where $\vf_\alpha\in L_{2,\text{loc}}(\bR^{d+1})$.
Then under Assumption \ref{assumption20080424} ($\gamma$),
for each $l \in \bZ$, $C \in \bC_l$, and $\kappa \ge 4$,
there exists a function $U^C$ depending on $C$ such that \eqref{eq12.20} is satisfied and \begin{equation}							\label{eq002}
\left(|U^C-(U^C)_{C}|\right)_{C}
\le N \left(F_{\kappa}\right)_C,		
\end{equation}
where 
$N=N(d,\delta,m,n,\tau)$
and
\begin{multline}							\label{eq001}
F_{\kappa}(t,x)=
\kappa^{-1/2}\big(\cM(U^2)\big)^{1/2}
\\
+\kappa^{m+\frac d2}\left[\big(\cM (F^2)\big)^{1/2}+\gamma^{1/(2\sigma)}
\big(\cM (U^{2\tau})\big)^{1/(2\tau)}\right].	
\end{multline}
\end{corollary}

\begin{proof}
For each $\kappa \ge 4$
and $C \in C_l$,
let $Q_r(X_0)$ be the smallest cylinder containing $C$.
From Theorem \ref{theorem3.42}, we fine $U^Q$ with $Q=Q_{\kappa r}(X_0)$. Take $U^C=U^Q$.
Then by Theorem \ref{theorem3.42}
as well as the facts that $C \subset Q_r(X_0)$ and the volumes of $C$ and $Q_r(X_0)$ are comparable,
we have
$$
\left(|U^C-(U^C)_{C}|\right)_{C}
\le N(d) I,
$$
where $I$ is the right hand side of the inequality \eqref{eq3.49}.
Note that, for example,
$$
\left(U^2\right)_{Q_{\kappa r}(X_0)}
\le \cM  (U^2) (X)
$$
for any $X=(t,x) \in C$.
Thus $I$ is less than
a constant times $F_{\kappa}(X)$ for any $X \in C$,
especially, it is less than a constant times $\left(F_{\kappa}\right)_C$.
Hence we arrive at the inequality \eqref{eq002}.
This finishes the proof of the corollary.
\end{proof}

\begin{theorem}							\label{theorem001}
Let $p \in (2,\infty)$, $\lambda > 0$,
and $\vf_{\alpha} \in L_p(\bR^{d+1})$.
There exist positive constants $\gamma \in (0,1)$ and $N$,
depending only on $d,\delta,m,n, p$, such that,
for $u \in C_0^{\infty}(\bR^{d+1})$ vanishing outside $Q_{\gamma R_0}$ and satisfying
$$
\vu_t+(-1)^m\cL \vu+\lambda \vu=\sum_{|\alpha|\le m}D^\alpha \vf_\alpha,
$$
we have
$$
\|U \|_{L_p(\bR^{d+1})}
\le N \|F\|_{L_p(\bR^{d+1})},
$$
where $N = N(d,\delta,m,n,p)$.
\end{theorem}

\begin{proof}
Let $\gamma>0$ and $\kappa\ge 4$ be constants to be specified below.
Take a constant $\tau$ such that $p > 2 \tau > 2$.
For each $l \in \bZ$ and $C \in \bC_l$, let $U^C$ be the function from Corollary \ref{cor001}.
We know that there exist positive $N_i$, $i=1,2$, depending only on $d$, $m$, $n$, and $\delta$, such that
$$
U\le N_1 U^C\le N_2 U=:V
$$
This along with Corollary \ref{cor001} and Theorem \ref{th081201} implies that
$$
\|U\|_{L_p}^p
\le N \|F_{\kappa}\|_{L_p}\|V\|_{L_p}^{p-1}
\le N \|F_{\kappa}\|_{L_p}\|U\|_{L_p}^{p-1}.
$$
Here we denote $L_p=L_p(\bR^{d+1})$.
The above inequalities readily give
\begin{equation}							\label{eq003}
\|U\|_{L_p} \le N \|F_{\kappa}\|_{L_p}.	
\end{equation}
From the definition of $F_{\kappa}$ \eqref{eq001} and the Hardy-Littlewood maximal function theorem
(recall that $p > 2\tau > 2$) it follows that
$$
\|F_{\kappa}\|_{L_p}
\le N \kappa^{-1/2}\|U\|_{L_p}
+N \kappa^{m+\frac d2}\|F\|_{L_p}
+ N \kappa^{m+\frac d2}\gamma^{1/(2\sigma)}\|U\|_{L_p}.	
$$
Using the above inequality and \eqref{eq003},
we have
$$
\|U\|_{L_p}
\le N \kappa^{-1/2}\|U\|_{L_p}
+N \kappa^{m+\frac d2}\|F\|_{L_p}
+ N \kappa^{m+\frac d2}\gamma^{1/(2\sigma)}\|U\|_{L_p}.
$$
It only remains to choose a sufficiently big $\kappa$, then a sufficiently small $\gamma$
so that
$$
N \kappa^{-1/2} + N \kappa^{m+\frac d2}\gamma^{1/(2\sigma)} < 1/2.
$$
\end{proof}

\begin{proof}[Proof of Theorem \ref{thm1}]
Due to the duality argument and Theorem \ref{theorem08061901}, it is enough to consider the case $p>2$. We first prove the first two assertions for $T=+\infty$.
In this case the estimate \eqref{eq080904}
is proved using Theorem \ref{theorem001} and the standard partition of unity argument. Then Assertion (ii) follows from the method of continuity and the existence of solutions to systems with simple coefficients,
for instance, $a_{\alpha\beta} = \delta_{\alpha\beta}I_{m \times m}$.
For $T<\infty$, we extend $\vf$ to be zero for $t\ge T$,
and then find a unique solution $\vu\in \mathring{\cH}^{m}_p(\Omega_\infty)$ of
\eqref{eq081902} in $\Omega_\infty$, the existence of which is guaranteed by the
argument above. This in turn also yields the existence of a solution of
\eqref{eq081902} in $\Omega_T$
satisfying \eqref{eq080904}. For the uniqueness, let $\vu \in \mathring{\cH}^{m}_p(\Omega_T)$ be a solution of \eqref{eq081902} with zero right-hand side in $\Omega_T$. Take $\tilde \vu$ to be the even extension of $\vu$ with respect to $t=T$. Then $\tilde \vu\in \mathring{\cH}^{m}_p(\Omega_\infty)$ satisfies \eqref{eq081902} in $\Omega_\infty$ with the right-hand side vanishing for $t<T$. It is easily seen from the method of continuity that $\tilde \vu\equiv 0$ for $t<T$.
Finally, Assertion (iii) is due to a standard scaling argument.
\end{proof}

\mysection{Boundary estimates}
                            \label{sec6}

This section is devoted to the Dirichlet problem for \eqref{eq0617_02}. We shall follow the lines of Sections \ref{sec_aux}--\ref{sec4} to carry out the corresponding boundary estimates. As before, we denote $\cL_0$ to be the operator with zero lower-order coefficients. We introduce a few more notation.
For any $t\in \bR$, $x\in \bR^d$ and $r>0$, denote
$$
\dist(x,\partial \Omega)=\inf_{y\in \partial\Omega}|x-y|,
\quad \Omega_r(x)=\Omega\cap B_r(x),
$$
$$
\cC_r(t,x)=(t-r^{2m},t)\times \Omega_r(x).
$$
$$
B_r^+(x)=B_r(x)\cap \bR^d_+,\quad
Q_r^+(t,x)=(t-r^{2m},t)\times B_r^+(x).
$$

\subsection{Boundary $L_2$-estimates}
Similar to Theorem \ref{theorem08061901}, we have the following $L_2$-estimate on a half space or on a domain.

\begin{theorem}			\label{thm6.1}
Let $T\in (-\infty,\infty]$ and $\Omega$ be a half space or a domain.\\
\noindent (i) There exists $N = N(d,m,n, \delta)$
such that, for any $\lambda \ge 0$,
\begin{equation}
                                \label{eq11.99}
\sum_{|\alpha|\le m}\lambda^{1-\frac {|\alpha|} {2m}} \|D^\alpha \vu \|_{L_2(\Omega_T)}
\le N \sum_{|\alpha|\le m}\lambda^{\frac {|\alpha|} {2m}} \| \vf_\alpha \|_{L_2(\Omega_T)},
\end{equation}
provided that $\vu \in {\mathring \cH}_2^m(\Omega_T)$, $\vf_\alpha \in L_2(\Omega_T),|\alpha|\le m$,
and in $\Omega_T$
\begin{equation}							 \label{eq2.27}
\vu_t+(-1)^m\cL_0 \vu +\lambda \vu = \sum_{|\alpha|\le m}D^\alpha \vf_\alpha.
\end{equation}\\
\noindent
(ii) For any $\lambda > 0$ and $\vf_\alpha \in L_2(\Omega_T)$, $|\alpha|\le m$, there exists a unique solution $\vu\in {\mathring \cH}^m_2(\Omega_T)$ to
the equation \eqref{eq2.27}. In the case that $\Omega$ is a bounded domain, one can take $\lambda=0$.
In this case, $N$ depends on $\Omega$ as well.
\end{theorem}
\begin{proof}
The proof is similar to that of Theorem \ref{theorem08061901}. The last assertion follows from the Poincar\'e inequality.
\end{proof}

In the proof of Proposition \ref{prop7.9}, we will extend Theorem \ref{thm6.1} to allow a more general right-hand side. In the remaining part of this section, we consider systems only on a half space.

First we prove the following boundary $L_2$-estimate, the proof of which is almost the same as that of Lemma \ref{lem2010_03} and thus omitted.

\begin{lemma}
                                            \label{lem6.2}
Let $0<r<R<\infty$. Assume $\vu\in C_{\text{loc}}^\infty(\overline{\bR^{d+1}_+})$
satisfies
\begin{equation}
                            \label{eq4.56}
D^k\vu =0\quad \text{on}\,\,\partial \bR^d_+,\quad k=0,1,\ldots,m-1,
\end{equation}
and
\begin{equation}
 					\label{eq2.54}
\vu_t+(-1)^m\cL_0 \vu=0
\end{equation}
in $Q_{R}^+$. Then there exists a constant $N=N(d,m,n,\delta)$ such that for $j=1,\ldots,m$,
$$
\|D^j\vu\|_{L_2(Q_r^+)}\leq N(R-r)^{-j}\|\vu\|_{L_2(Q_R^+)}.
$$
\end{lemma}

\begin{corollary}
                                    \label{cor6.3}
Let $0<r<R<\infty$ and $a_{\alpha\beta}=a_{\alpha\beta}(x_1)$, $|\alpha|=|\beta|=m$. Assume that $\vu\in C_{\text{loc}}^\infty(\overline{\bR^{d+1}_+})$ satisfies
\eqref{eq4.56} and \eqref{eq2.54}
in $Q_R^+$. Then for any integers $i\ge 1$, $j\ge 0$, and any $\alpha$ satisfying $\alpha_1\le m$, we have
\begin{equation*}
\|\partial_t^i \vu\|_{L_2(Q_r^+)}
+\|\partial_t^j D^\alpha\vu\|_{L_2(Q_r^+)}\le N\|D^m\vu\|_{L_2(Q_R^+)},
\end{equation*}
where $N=N(d,m,n,\delta, R, r, \alpha, i,j)$.
\end{corollary}
\begin{proof}
Since $D_{x'}^{\alpha'} \vu$ also satisfies \eqref{eq4.56} and \eqref{eq2.54}, by applying Lemma \ref{lem6.2} repeatedly, we obtain for any $\alpha$ satisfying $\alpha_1\le m$,
\begin{equation*}
\|D^\alpha\vu\|_{L_2(Q_r^+)}\le N\|\vu\|_{L_2(Q_{R'}^+)},
\end{equation*}
where $R'=(r+R)/2$.
From this inequality and the boundary version of the Poincar\'e inequality along with the zero boundary condition \eqref{eq4.56}  it follows that
$$
\|D^\alpha\vu\|_{L_2(Q_r^+)}\le N\|D^m\vu\|_{L_2(Q_R^+)}.
$$
The get the desired estimate, it suffices to use the argument in Lemma \ref{lem2010_1}.
\end{proof}

Recall that
$$
\Theta:=\sum_{|\beta|=m}a_{\bar\alpha\beta}D^\beta \vu,\quad \bar\alpha=me_1.
$$

\begin{lemma}
 				\label{lem6.6}
Let $0<r<R<\infty$ and $a_{\alpha\beta}=a_{\alpha\beta}(x_1)$.
Assume $\vu\in C_{\text{loc}}^\infty(\overline{\bR^{d+1}_+})$ satisfies \eqref{eq4.56} and
\eqref{eq2.54}
in $Q_R^+$.
Then, for any nonnegative integers $i$ and $j$,
\begin{equation*}	
\|\partial_t^iD^j_{x'} \Theta\|_{L_2(Q_r^+)}
+ \|\partial_t^i D^j_{x'} D_1 \Theta \|_{L_2(Q_r^+)}
\le N \|D^m\vu\|_{L_2(Q_R^+)},
\end{equation*}
where $N=N(d, m,n,r,R,\delta, i, j)>0$.
\end{lemma}

\begin{proof}
Obviously, we have
\begin{equation}
                                    \label{eq11.07b}
\|\Theta \|_{L_2(Q_r^+)}
\le N \|D^m\vu\|_{L_2(Q_r^+)}.
\end{equation}
As before, it suffices to prove that, for $R'=(r+R)/2$,
\begin{equation}
                        \label{eq0924}
\|D_1 \Theta \|_{L_2(Q_r^+)}
\le N \|D^m\vu\|_{L_2(Q_{R'}^+)}.
\end{equation}
From \eqref{eq2.54}, in $Q_R^+$ we have
$$
D_1^m\Theta=(-1)^{m+1}\vu_t-\sum_{\substack{|\alpha|=|\beta|=m\\\alpha_1<m}}D_\alpha(a_{\alpha\beta}D_\beta \vu)
$$
$$
=(-1)^{m+1}\vu_t-\sum_{\substack{|\alpha|=|\beta|=m\\\alpha_1<m}}D_1^{\alpha_1}
(a_{\alpha\beta}D_{x'}^{\alpha'}D^\beta \vu).
$$
Then the estimate \eqref{eq0924} follows from Corollary \ref{corA.2} with a covering argument and Corollary \ref{cor6.3}. The lemma is proved.
\end{proof}

\subsection{Boundary H\"older estimates and Hardy's inequality}
In the same fashion as in Section \ref{sec3.2}, we obtain

\begin{lemma}
 				\label{lem6.4}
Let $a_{\alpha\beta}=a_{\alpha\beta}(x_1)$.
Assume that $\vu\in C_{\text{loc}}^\infty(\overline{\bR^{d+1}_+})$ satisfies
\eqref{eq2.54} in $Q_2^+$ and \eqref{eq4.56}.
Then for any $\alpha$ satisfying $|\alpha|=m$ and $\alpha_1<m$, we have
\begin{equation*}	
\|D^\alpha\vu\|_{C^{1/2}(Q_1^+)}
\le N \|D^m\vu\|_{L_2(Q_2^+)},
\end{equation*}
where $N=N(d,m,n,\delta)>0$.
\end{lemma}
\begin{proof}
We follow the proof of Lemma \ref{lemma100} by using Corollary \ref{cor6.3} instead of Lemma \ref{lem2010_1}.
\end{proof}

\begin{corollary}
 				\label{cor6.5}
Let $a_{\alpha\beta}=a_{\alpha\beta}(x_1)$ and $\lambda\ge 0$.
Assume that $\vu\in C_{\text{loc}}^\infty(\overline{\bR^{d+1}_+})$ satisfies \eqref{eq4.56} and
\begin{equation*}
\vu_t+(-1)^m\cL_0 \vu+\lambda \vu=0
\end{equation*}
in $Q_2^+$.
Then we have
\begin{equation*}	
\|U'\|_{C^{1/2}(Q_1^+)}
\le N \|U\|_{L_2(Q_2^+)},
\end{equation*}
where $N=N(d,m,n,\delta)>0$, and $U,U'$ are defined  as in \eqref{eq10_100}.
\end{corollary}

\begin{proof}
The corollary follows
from Lemma \ref{lem6.4} and Agmon's idea as in the proof of Corollary \ref{cor3.5}.
\end{proof}

As an analogy of Lemma \ref{lemma3.6}, we obtain

\begin{lemma}
 				\label{lem6.7}
Under the conditions of Corollary \ref{cor6.5}, we have
\begin{equation*}	
\|\Theta\|_{C^{1/2}(Q_1^+)}
\le N \|U\|_{L_2(Q_2^+)},
\end{equation*}
where $N=N(d,m,n,\delta)>0$.
\end{lemma}

Similar to Corollary \ref{cor3.10}, Lemma \ref{lem6.7} and Corollary \ref{cor6.5} imply
\begin{corollary}
                                        \label{cor3.10b}
Under the conditions of Corollary \ref{cor6.5}, we have
\begin{equation*}	
\|U\|_{L_\infty(Q_1^+)}
\le N \|U\|_{L_2(Q_2^+)},
\end{equation*}
where $N=N(d,m,n,\delta)>0$.
\end{corollary}

We will use the following Hardy type inequality.

\begin{lemma}
                            \label{gHardy}
Let $R\in (0,\infty)$, $p\in (1,\infty]$, $m$ be a positive integer, and $f\in C^\infty([0,R])$. Suppose that
$f(0)=Df(0)=...=D^{m-1} f(0)$. Then we have
\begin{equation}
                                    \label{eq10.11}
\|x^{-k}D^{m-k}f(x)\|_{L_p([0,R])}\le N\|D^mf\|_{L_p([0,R])},
\quad k = 1, \ldots, m,
\end{equation}
where $N=N(p,k)>0$ is a constant. Furthermore, for any function $g\in L_q([0,R])$, $q=p/(p-1)$, and $\eta\in C_{\text{loc}}^\infty((0,R])$ satisfying
$$
|D^k \eta(x)|\le Kx^{-k},\quad x\in (0,R],\,\,k=0,1,\ldots,m,\,\,K>0,
$$
we have
$$
\|D^m(\eta f)\|_{L_p([0,R])}\le N\|D^m f\|_{L_p(\text{supp}\,\eta)},
$$
$$
\|g D^m(\eta f)\|_{L_1([0,R])}\le N\|g\|_{L_{q}(\text{supp}\,\eta)}\|D^m f\|_{L_p(\text{supp}\,\eta)},
$$
where $N=N(p,m,K)>0$.
\end{lemma}

\begin{proof}
In the case $k=1$, the inequality \eqref{eq10.11} is the classical Hardy's inequality. For $1<k\le m$, the inequality is known in one form or another.
Here we give a simple proof for completeness. By Taylor's formula, for any $x\in (0,R]$, we have
$$
x^{-k}D^{m-k}f(x)=\big(\Gamma(k)\big)^{-1}\int_0^1 (1-r)^{k-1}(D^m f)(rx)\,dr,
$$
where $\Gamma$ is the Gamma function.
Thanks to Young's inequality, we get
$$
\|x^{-k}D^{m-k}f(x)\|_{L_p([0,R])}\le
\big(\Gamma(k)\big)^{-1}\int_0^1 (1-r)^{k-1}\|(D^m f)(rx)\|_{L_p([0,R])}\,dr
$$
$$
\le \big(\Gamma(k)\big)^{-1}\int_0^1 (1-r)^{k-1}r^{-1/p}\,dr \|D^m f\|_{L_p([0,R])}
$$
$$
=\Gamma(1-1/p)\big(\Gamma(k+1-1/p)\big)^{-1}\|D^m f\|_{L_p([0,R])}.
$$
The first assertion is proved. The second assertion follows easily from the first one by using the Leibniz rule and H\"older's inequality.
\end{proof}

%

\subsection{Estimates of mean oscillations}
Now we prove the following estimate of mean oscillations.
As in Section \ref{sec5}, we assume that all the lower-order coefficients of $\cL$ are zero.

\begin{proposition}
 				\label{prop7.9}
Let $t_0\in \bR$, $x_0\in \overline{\bR^d_+}$, $X_0=(t_0,x_0)$, $r\in (0,\infty)$, $\kappa\in [64,\infty)$, $\lambda\ge 0$, $\nu \in (2,\infty)$, $\nu'=2\nu/(\nu-2)$, and
$\vf_\alpha= (f_\alpha^1, \ldots, f_\alpha^n)^{\text{tr}} \in L_{2,\text{loc}}(\overline{\bR^{d+1}_+})$.
Assume that $\kappa r\le R_0$ and $\vu\in C_{\text{loc}}^\infty(\overline{\bR^{d+1}_+})$ satisfies \eqref{eq4.56} and
\begin{equation*}
\vu_t+(-1)^m\cL \vu+\lambda \vu=\sum_{|\alpha|\le m}D^\alpha \vf_\alpha
\end{equation*}
in $Q^+_{\kappa r}(X_0)$. Then under Assumption \ref{assumption20100901} ($\gamma$),
there exists a function $U^Q$ depending on $Q^+:=Q^+_{\kappa r}(X_0)$ such that $N^{-1}U\le U^Q\le NU$ and
$$
\big(|U^Q-(U^Q)_{Q_r^+(X_0)}|\big)_{Q_r^+(X_0)}
\le N(\kappa^{-1/2}+\kappa \gamma) \big(U^2\big)_{Q^+_{\kappa r}(X_0)}^{1/2},
$$
\begin{equation}	
 			\label{eq5.42}
+N\kappa^{m+\frac d2}\left[(F^2)_{Q^+_{\kappa r}(X_0)}^{1/2}+\gamma^{1/\nu'}
(U^\nu)_{Q^+_{\kappa r}(X_0)}^{1/\nu}\right],
\end{equation}
where $N=N(d,m,n,\delta,\nu)>0$.
\end{proposition}

The proof of the proposition is split into two cases.
As in the proof of Lemma \ref{lem4.2}, we assume that coefficients $a_{\alpha\beta}$ and $\vf_{\alpha}$ are infinitely differentiable.

{\em Case 1: the first coordinate of $x_0$ $\ge \kappa r/16$.} In this case, we have
$$
Q_r^+(X_0)=Q_r(X_0)\subset Q_{\kappa r/16}(X_0)\subset \bR^{d+1}_+.
$$
Since $\kappa/16\ge 4$, \eqref{eq5.42} follows immediately
by applying Theorem \ref{theorem3.42}  with $\frac \kappa {16}$ in place of $\kappa$.
Note that Theorem \ref{theorem3.42} is proved under the assumption that $\vu$ vanishes outside $Q_{\gamma R_0}$. However, one can see that the proof of the theorem does not use this assumption in the case $\kappa r < R_0$, more precisely, $\frac \kappa {16} r < R_0$.

{\em Case 2: $0\le$ the first coordinate of $x_0$ $< \kappa r/16$.}
Denote $y_0=(0,x_0')$ and $Y_0=(t_0,y_0)$. Without loss of generality, one may assume $Y_0=(0,0)$. Notice that in this case,
\begin{equation}
                                \label{eq22.15}
Q_r^+(X_0)\subset Q^+_{\kappa r/8} \subset Q^+_{\kappa r/4}\subset Q^+_{\kappa r/2}
\subset Q^+_{\kappa r}(X_0).
\end{equation}
Denote $R=\kappa r/2(< R_0)$. Because of Assumption \ref{assumption20100901}, after an orthogonal transformation $y = \rho x$ centered at $Y_0=(0,0)$, we may assume
$$
 \{(y_1,y'):\gamma R<y_1\}\cap B_R
 \subset\Omega\cap B_{R}
 \subset \{(y_1,y'):-\gamma R< y_1\}\cap B_{R},
$$
where $\Omega$ is the image of $\bR^d_+$ under the orthogonal transformation,
and
\begin{equation}
								\label{eq17_50}
\sup_{|\alpha|=|\beta|=m}\int_{Q_R} |a_{\alpha\beta}(t,y) - \bar{a}_{\alpha\beta}(y_1)| \, dx \, dt \le \gamma |Q_R|.
\end{equation}
Let $\tilde X_0$ be the new coordinates of $X_0$ after the orthogonal transformation. Then \eqref{eq22.15} becomes
\begin{equation}
                                \label{eq22.16}
\cC_r(\tilde X_0)\subset \cC_{R/4} \subset \cC_{R/2}\subset \cC_{R}
\subset \cC_{\kappa r}(\tilde X_0).
\end{equation}
Without any confusion, in the new coordinate system we still denote the corresponding unknown function, the coefficients, and the data by $\vu$, $a_{\alpha\beta}$, and $\vf$, respectively.
Take a smooth function $\chi$ defined on $\bR$ such that
$$
\chi(y_1)\equiv 0\quad\text{for}\,\,y_1\le \gamma R,
\quad \chi(y_1)\equiv 1\quad\text{for}\,\,y_1\ge  2\gamma R,
$$
$$
|D^k \chi|\le N(\gamma R)^{-k}\quad\text{for}\,\,k=1,2,...,m.
$$

Below we present a few lemmas, which should be read as parts of the proof of the second case.

\begin{lemma}
Let $\hat \vu:=\chi\vu$.
Then $\hat \vu$ along with all its derivatives vanishes on $Q_R \cap \{y_1 \le \gamma R\}$ and satisfies in $Q_{R}^{\gamma+}:=Q_{R}\cap \{y_1> \gamma R\}$,
$$
\hat \vu_t+(-1)^m \cL_0 \hat\vu+\lambda \hat\vu=(-1)^m \sum_{|\alpha|=|\beta|=m}D^\alpha\left(
(\bar a_{\alpha\beta}- a_{\alpha\beta})D^\beta \vu\right)
$$
\begin{equation}
                                    \label{eq17.23b}
+ \sum_{|\alpha|\le m}\chi D^\alpha \vf_\alpha+(-1)^m\vg+(-1)^m\vh,
\end{equation}
where $\cL_0$ is the differential operator with the coefficients $\bar a_{\alpha\beta}$ from \eqref{eq17_50}, and
\begin{align*}
\vg&=\sum_{|\alpha|=|\beta|=m}
D^\alpha\big(\bar a_{\alpha\beta} D^{\beta}((\chi-1)\vu)\big),\\
\vh&=(1-\chi)\sum_{|\alpha|=|\beta|=m}
D^\alpha( a_{\alpha\beta} D^{\beta}\vu).
\end{align*}	
\end{lemma}

\begin{proof}
This can be easily seen if one begins with multiplying the equation of $\vu$ by $\chi$ and then adding $(-1)^m \vh$ to the both sides.	
\end{proof}

Now let $\hat\vw$ be the unique $\mathring\cH^{m}_2(\bR\times \{y:y_1>\gamma R\})$ solution of
$$
\hat\vw_t+(-1)^m \cL_0 \hat\vw+\lambda \hat\vw
=(-1)^m \sum_{|\alpha|=|\beta|=m}D^\alpha\left(
\varphi(\bar a_{\alpha\beta}- a_{\alpha\beta})D^\beta \vu\right)
$$
\begin{equation}
                        \label{eq17.25}
+(-1)^m\hat \vg
+ \sum_{|\alpha|\le m}\chi D^\alpha (\varphi\vf_\alpha)+(-1)^m \hat\vh
\end{equation}
in $\bR\times \{y:y_1>\gamma R\}$, where
$\varphi$ is an infinitely differentiable function such that
$$
0 \le \varphi \le 1,
\quad
\varphi = 1 \,\, \text{on} \,\, Q_{R/2},
\quad
\varphi = 0 \,\, \text{outside} \,\, (-R^{2m}, R^{2m}) \times B_R,
$$
and
\begin{align*}
\hat \vg&=\sum_{|\alpha|=|\beta|=m}
D^\alpha\big(\bar a_{\alpha\beta} \varphi D^{\beta}((\chi-1)\vu)\big),\\
\hat\vh&=(1-\chi)\sum_{|\alpha|=|\beta|=m}
D^\alpha( a_{\alpha\beta} \varphi D^{\beta}\vu).
\end{align*}
Note that by the classical theory $\hat{\vw}$ is infinitely differentiable in $\bR \times \{y: y_1 > \gamma R\}$.

\begin{lemma}
								\label{lem0927}
Let $T \in (-\infty,\infty]$ and $\Omega' = \{y \in \bR^d: y_1 > \gamma R\}$.
Then for the function $\hat \vw$ in \eqref{eq17.25}, we have
\begin{multline*}	
\sum_{|\alpha|\le m}\lambda^{\frac 1 2-\frac {|\alpha|} {2m}}\|D^{\alpha} \hat \vw\|_{L_2(\Omega'_T)}
\le N \sum_{|\alpha|=|\beta|=m}\|\varphi(\bar a_{\alpha\beta}- a_{\alpha\beta})D^\beta \vu\|_{L_2(\Omega'_T)}
\\
+N \sum_{|\beta|=m}\|\varphi D^{\beta}((\chi-1)\vu)\|_{L_2(\Omega'_T)}
+N\sum_{|\alpha|\le m}\lambda^{\frac{|\alpha|}{2m}-\frac 1 2} \|\varphi\vf_\alpha\|_{L_2(\Omega'_T)}
\\
+N \sum_{|\beta|=m}\|\varphi D^{\beta}\vu\|_{L_2(\Omega'_T \cap \{(t,y): \gamma R < y_1 < 2 \gamma R\})},
\end{multline*}
where $N=N(d,m,n,\delta)$.
\end{lemma}

\begin{proof}
The first two terms on the right-hand side result from a direct application of Theorem \ref{thm6.1} with $\Omega'$ in place of $\bR^d_+$. For the third term, due to the presence of the factor $\chi$, one cannot directly apply Theorem \ref{thm6.1} to get the desired estimate.
However, we observe that as in the proof of Theorem \ref{theorem08061901}, after testing the equation by $\hat \vw$ and integrating by parts, the terms $D^\alpha \vu$ on the right-hand side of \eqref{eq16.33} are now replaced by $D^\alpha(\chi \hat\vw)$.
Notice that
$$
|D^k \chi|\le N(y_1-\gamma R)^{-k}\quad\text{for}\,\,k=1,2,...,m.
$$
Then by Lemma \ref{gHardy}, the $L_2$ norm of $D^\alpha(\chi \hat\vw)$ is bounded above by the $L_2$ norm of $D^\alpha \hat\vw$, which implies that the same estimate as in Theorem \ref{thm6.1} still holds true even with the presence of $\chi$.
For the last term in the above estimate, we argue in the same way and also use the property that $\chi - 1$ is supported on $\{y: y_1 \le 2\gamma R\}$.
\end{proof}

Recall that
$$
U=\sum_{|\alpha| \le m}\lambda^{\frac 1 2-\frac {|\alpha|}{2m}} |D^\alpha \vu|,
\quad
F=\sum_{|\alpha|\le m}\lambda^{\frac {|\alpha|} {2m}-\frac 1 2}|\vf_\alpha|.
$$

\begin{lemma}
For the function $\hat \vw$ in \eqref{eq17.25}, we have
\begin{equation}	
 			\label{eq21.52h}
\sum_{k=0}^m\lambda^{\frac 1 2-\frac k {2m}}(I_{Q_{R}^{\gamma +}}|D^k \hat \vw|^2)_{\cC_{R}}^{1/2}
\le N\gamma^{1/ {\nu'}} (U^\nu)_{\cC_R}^{1/ \nu}+
N(F^2)_{\cC_{R}}^{1/2},
\end{equation}
where $\nu$ and $\nu'$ are from Proposition \ref{prop7.9}.
\end{lemma}

\begin{proof}
Note that $\varphi(t,y)$ vanishes on $\{y: |y| \ge R, y_1 > \gamma R\}$.
Thus by the estimate in Lemma \ref{lem0927} when $T=0$, it follows that
the left-hand side 	of \eqref{eq21.52h} is less than a constant times
$$
\sum_{|\alpha|=|\beta|=m}\left(|(\bar a_{\alpha\beta}- a_{\alpha\beta})D^\beta \vu|^2\right)^{1/2}_{\cC_R}
+ \sum_{|\beta|=m}\left(I_{Q_R^{\gamma +}}|D^{\beta}((\chi-1)\vu)|^2\right)^{1/2}_{\cC_R}
$$
$$
+ \sum_{|\alpha|\le m}\lambda^{\frac{|\alpha|}{2m}-\frac 1 2} \left(|\vf_\alpha|^2\right)^{1/2}_{\cC_R}
+ \sum_{|\beta|=m}\left(I_{\{\gamma R <y_1< 2\gamma R\}}|D^{\beta}\vu|^2\right)^{1/2}_{\cC_R}
:=I_1+I_2+I_3+I_4.
$$
It is clear that $I_3$ is bounded by $N(F^2)_{\cC_{R}}^{1/2}$.
By using \eqref{eq17_50} as in the interior case we see that $I_1$ is bounded by $N\gamma^{1/ {\nu'}} (U^\nu)_{\cC_R}^{1/ \nu}$.
Observe that by H\"older's inequality we have
\begin{multline}
                                    \label{eq21.44h}
\left(I_{\{\gamma R < y_1 < 2\gamma R\}}
|D^{m}\vu|^2\right)^{1/2}_{\cC_R}
\le \left(I_{\{\gamma R < y_1 < 2 \gamma R\}}\right)_{\cC_R}^{1/\nu'}
\left(|D^m\vu|^\nu\right)^{1/\nu}_{\cC_R}
\\
\le N\gamma^{1/\nu'}(|D^m\vu|^{\nu})_{\cC_R}^{1/\nu}.
\end{multline}
Thus $I_4$ is also bounded by $N\gamma^{1/ {\nu'}} (U^\nu)_{\cC_R}^{1/ \nu}$.

To estimate $I_2$, we notice that $\chi-1=0$ for $y_1\ge 2\gamma R$, $D^{\beta-\hat \beta}\chi=0$ in $(-\infty,\gamma R]\cap [2\gamma R,+\infty)$, and
\begin{equation}
                                    \label{eq21.13h}
|D^{\beta-\hat \beta}\chi|\le N(\gamma R)^{-|\beta|+|\hat \beta|}
\end{equation}
in $(\gamma R,2\gamma R)$, where $|\beta|=m$ and $|\hat \beta| \le m-1$.
For any $(s,\gamma R, y')\in Q_R$, let $\hat y_1=\hat y_1(s,y')$ be the largest number such that $(s,\hat y_1,y')\in \partial\Omega$. (Indeed, in this case $\hat y_1$ is uniquely determined as a function of $y'$.) Because $\hat y_1\in (-\gamma R,\gamma R)$, the inequality \eqref{eq21.13h} implies
\begin{equation}
								\label{eq0929}
|D^{\beta-\hat \beta}\chi(s,y)|\le N(y_1-\hat y_1)^{-|\beta|+|\hat \beta|}.
\end{equation}
We also notice that $\vu$, as a function of $y_1$, vanishes along with its derivatives up to $(m-1)$-th order at $(s,\hat y_1,y')$.
Thus, by Lemma \ref{gHardy} together with \eqref{eq0929},
\begin{multline}
								\label{eq0928}
\int_{\gamma R}^r |D^{\beta}\big((\chi-1)\vu(s,y_1,y')\big)|^2\,dy_1
\\
\le \int_{\hat y_1}^r |D^{\beta}\big((\chi-1)\vu(s,y_1,y')\big)|^2\,dy_1
\le N\int_{\hat y_1}^r |D^{m}\vu(s,y_1,y')|^2\,dy_1,
\end{multline}
where $r=r(y')=\min\{2 \gamma R,\sqrt{R^2-|y'|^2}\}$.
Integrating with respect to $s$ and $y'$ and using H\"older's inequality as in \eqref{eq21.44h}, we estimate $I_2$ by
\begin{equation*}
\big(I_{Q_R^{\gamma+}}|D^{\beta}((\chi-1)\vu)|^2\big)_{\cC_R}^{1/2}
\le N\gamma^{1/ {\nu'}}(|D^m\vu|^{\nu})_{\cC_R}^{1/ \nu}.
\end{equation*}
From this and the above estimates for $I_i$, $i=1,3,4$, we conclude \eqref{eq21.52h}.
\end{proof}

Now we are ready to complete the proof of Proposition \ref{prop7.9}.
We extend $\hat \vw$ to be zero in $\cC_R\setminus Q_R^{\gamma +}$, so that $\hat\vw\in \cH^{m}_2(\cC_R)$, and let $\vw=\hat \vw+(1-\chi)\vu$.
Since $(1-\chi)\vu$ vanishes for $y_1 \ge 2 \gamma R$,
using the second inequality in \eqref{eq0928} and H\"older's inequality as in \eqref{eq21.44h}, we see that
$$
\sum_{k=0}^m\lambda^{\frac 1 2-\frac k {2m}}(|D^k ((1-\chi)\vu)|^2)_{\cC_{R}}^{1/2}
\le N\gamma^{1/ {\nu'}} (U^\nu)_{\cC_R}^{1/ \nu}.
$$
This combined with \eqref{eq21.52h} shows that
\begin{equation}	
 			\label{eq18.34h}
(W^2)_{\cC_{R}}^{1/2}
\le N\gamma^{1/ {\nu'}} (U^\nu)_{\cC_{R}}^{1/ \nu}+
N(F^2)_{\cC_{R}}^{1/2}.
\end{equation}
Noting that $\cC_r(\tilde X_0) \subset \cC_R$ and $|\cC_R|/|\cC_r(\tilde X_0)|
\le N(d) \kappa^{2m+d}$, we obtain from \eqref{eq18.34h}
\begin{equation}
								\label{eq28_01}
(W^2)_{\cC_r(\tilde X_0)}^{1/2}
\le N\kappa^{m+\frac d 2} \left(\gamma^{1/ {\nu'}} (U^\nu)_{\cC_{R}}^{1/ \nu}+
(F^2)_{\cC_{R}}^{1/2}\right).
\end{equation}

Next, we define $\vv=\vu-\vw$ in $\cC_R$. It is easily seen that $\vv=0$ in $\cC_R\setminus Q_R^{\gamma +}$ and $\vv$ satisfies
$$
\vv_t+(-1)^m \cL_0 \vv+\lambda \vv=0	
$$
in $Q_{R/2} \cap \{y_1> \gamma R\}$ and vanishes along with its derivatives up to $(m-1)$-th order on $Q_R \cap \{y_1=\gamma R\}$.
Denote
$$
\cD_1=\cC_{r}(\tilde X_0)\cap \{y_1<\gamma R\},
\quad
\cD_2=\cC_{r}(\tilde X_0)\setminus \cD_1,
\quad
\cD_3=Q_{R/4}\cap\{y_1 > \gamma R\}.
$$
Because of \eqref{eq22.16},
$
|\cD_1|\le N\kappa\gamma|\cC_{r}(\tilde X_0)|
$.
Thanks to the fact that $\hat \vw$ is infinitely differentiable in $\bR \times \{y: y_1 > \gamma R\}$,
we see that $\vv$ is infinitely differentiable in $Q_R \cap \{y_1 > \gamma R\}$. Then applying Corollary \ref{cor6.5} and Lemma \ref{lem6.7} with a scaling argument, we compute
$$
\big(|V'-(V')_{\cC_r(\tilde X_0)}|\big)_{\cC_r(\tilde X_0)}
+\big(|\hat \Theta-(\hat \Theta)_{\cC_r(\tilde X_0)}|\big)_{\cC_r(\tilde X_0)}
$$
$$
\le N r^{1/2}\big([V']_{C^{1/2}(\cD_2)}
+[\hat \Theta]_{C^{1/2}( \cD_2)}\big)
+N\kappa\gamma \|V\|_{L_\infty(\cD_2)}
$$
$$
\le N r^{1/2}\big([V']_{C^{1/2}(\cD_3)}
+[\hat \Theta]_{C^{1/2}( \cD_3)}\big)
+N\kappa\gamma \|V\|_{L_\infty(\cD_3)}
$$
$$
\le N(\kappa^{-1/2}+\kappa \gamma)(V^2)_{\cC_{R/2}}^{1/2},
$$
which together with \eqref{eq18.34h} and \eqref{eq28_01} yields \eqref{eq5.42}.
Indeed,
we set $U^Q=U'+\Theta$.
Then
$$
\big(|U^Q-(U^Q)_{\cC_r(\tilde X_0)}|\big)_{\cC_r(\tilde X_0)}
\le N\big(|V'-(V')_{\cC_r(\tilde X_0)}|\big)_{\cC_r(\tilde X_0)}
$$
$$
+ N \big(|\hat \Theta-(\hat \Theta)_{\cC_r(\tilde X_0)}|\big)_{\cC_r(\tilde X_0)}
+ N \big(W\big)_{\cC_r(\tilde X_0)},
$$
where the last term is estimated by \eqref{eq28_01}.
As shown above, the first two terms on the right-hand side are estimated by
$N(\kappa^{-1/2}+\kappa \gamma)(V^2)_{\cC_{R/2}}^{1/2}$, which is taken care of by \eqref{eq18.34h} and the fact that $\vu = \vv+\vw$ in $\cC_R$.
Finally, we transform the obtained inequality back to the original coordinates to get the inequality \eqref{eq5.42}.
This completes the proof of Proposition \ref{prop7.9}.

\subsection{Proof of Theorem \ref{thm3}}
We finish the proof of Theorem \ref{thm3} in this subsection. As in Section \ref{sec5}, we assume all the lower-order coefficients of $\cL$ are zero.

Next in the measure
space $\bR^{d+1}_{+}$ endowed with the Borel $\sigma$-field and Lebesgue measure consider the filtration of dyadic parabolic cubes $\{\bC_{l},l\in\bZ\}$, where
$\bZ=\{0,\pm1,\pm2,...\}$ and $\bC_l$ is the collection of cubes
$$
(i_{0}2^{-2ml},(i_{0}+1)2^{-2ml}]\times
 (i_{1}2^{-l},(i_{1}+1)2^{-l}]\times...\times
(i_{d}2^{-l},(i_{d}+1)2^{-l}],
$$
where $i_{0},i_{1},...,i_{d}\in\bZ,\,\,i_1\ge 0$.
Notice that if $X\in C\in \bC_l$, then for the smallest $r>0$
such that $C\subset Q_{r}(X)$ we have
$$
\dashint_{C} \dashint_{C}|g(Y)-g(Z)|
\,dY\,dZ\leq N(d)
\dashint_{Q^{+}_{r}(X)} \dashint_{Q^{+}_{r}(X)}|g(Y)-g(Z)|
\,dY\,dZ.
$$
By using this, the following corollary is proved in the same manner as Corollary \ref{cor001}.

\begin{corollary}							\label{cor001b}
Let $\gamma \in (0,1/4)$, $\lambda > 0$, $\nu\in (2,\infty)$, $\nu'=2\nu/(\nu-2)$, and
$Z_0 \in \overline{\bR^{d+1}_+}$.
Assume that $\vu\in C_{\text{loc}}^\infty(\overline{\bR^{d+1}_+})$ vanishes outside $Q_{\gamma R_0}(Z_0)$ and
satisfies \eqref{eq4.56} as well as
$$
\vu_t+(-1)^m\cL \vu+\lambda \vu=\sum_{|\alpha|\le m}D^\alpha \vf_\alpha
$$
in $\bR^{d+1}_+$,
where $\vf_\alpha\in L_{2,\text{loc}}(\overline{\bR^{d+1}_+})$.
Then under Assumption \ref{assumption20100901} ($\gamma$),
for each $l \in \bZ$, $C \in \bC_l$, and $\kappa \ge 64$,
there exists a function $U^C$ depending on $C$
such that $N^{-1}U\le U^C\le NU$ and
$$
\left(U^C-(U^C)_{C}|\right)_{C}
\le N \left(F_{\kappa}\right)_C,		
$$
where
$N=N(d,\delta,m,n,\tau)$ and
$$
F_{\kappa}=
(\kappa^{-1/2}+\kappa \gamma)\big(\cM ( \hat U^2)\big)^{1/2}
+\kappa^{m+d/2}\left[\big(\cM(\hat F^2)\big)^{1/2}+\gamma^{1/\nu'}
(\cM(U^{\nu}))^{1/\nu}\right].	
$$
\end{corollary}

\begin{theorem}							\label{theorem001b}
Let $p \in (2,\infty)$, $\lambda > 0$, $Z_0 \in \overline{\bR^{d+1}_+}$,
and $\vf_{\alpha} \in L_p(\bR^{d+1}_+)$.
There exist positive constants $\gamma\in (0,1/4)$ and $N$,
depending only on $d$, $\delta$, $m$, $n$, $p$, such that under Assumption \ref{assumption20100901} ($\gamma$),
for $\vu \in C_{\text{loc}}^\infty(\overline{\bR^{d+1}_+})$ vanishing outside $Q_{\gamma R}(Z_0)$ and satisfying
\eqref{eq4.56} as well as
$$
\vu_t+(-1)^m\cL \vu+\lambda \vu=\sum_{|\alpha|\le m}D^\alpha \vf_\alpha
$$
in $\bR^{d+1}_+$,
we have
$$
\|U\|_{L_p(\bR^{d+1}_+)}
\le N \| F\|_{L_p(\bR^{d+1}_+)},
$$
where $N = N(d,\delta,m,n,p)$.
\end{theorem}

\begin{proof}
It follows from the proof of Theorem \ref{theorem001} by using Corollary \ref{cor001b} and Theorem \ref{th081201}.	
\end{proof}

\begin{proof}[Proof of Theorem \ref{thm3}]
Due to the duality argument and Theorem \ref{thm6.1}, it is enough to consider the case $p>2$.
In this case the theorem is proved using Theorem \ref{theorem001b} and the standard partition of unity argument.
\end{proof}

\mysection{Systems on Reifenberg flat domains}
							\label{Reifenberg}

As an application, in this section we consider systems on a cylindrical domain $(0,T)\times \Omega$. Here $T\in (0,\infty)$ and $\Omega$ is a Reifenberg flat domain in $\bR^d$, which is not necessarily to be bounded. Roughly speaking, the boundary of a Reifenberg flat domain is locally trapped in thin discs. 


We impose a similar regularity assumption on $a_{\alpha\beta}$ as Assumption \ref{assumption20080424}. Near the boundary, we require that in each small scale the direction in which the coefficients are only measurable coincides with the ``normal'' direction of a certain thin disc, which contains a portion of $\partial\Omega$. More precisely, we assume the following, where the parameter $\gamma\in (0,1/20)$ will be determined later.
\begin{assumption}[$\gamma$]
                                        \label{assump1}
There is a constant $R_0\in (0,1]$ such that the following holds.

(i) For any $x\in \Omega$, $t\in \bR$, and any $r\in (0,\min\{R_0,\dist(x,\partial\Omega)/2\}]$ (so that $B_r(x)\subset \Omega$), there is a spatial coordinate system depending on $(t,x)$ and $r$ such that in this new coordinate system, we have
\begin{equation}
                            \label{eq13.07}
\dashint_{Q_r(t,x)}\Big| a_{\alpha\beta}(s,y_1, y') - \dashint_{Q'_r(t,x')} a_{\alpha\beta}(\tau,y_1,z') \, dz'\,d\tau \Big| \, dy\,ds\le \gamma.
\end{equation}

(ii) For any $x\in \partial\Omega$, $t\in \bR$, and any $r\in (0,R_0]$, there is a spatial coordinate system depending on $(t,x)$ and $r$ such that in this new coordinate system, we have \eqref{eq13.07} and
$$
 \{(y_1,y'):x_1+\gamma r<y_1\}\cap B_r(x)
 \subset\Omega_r(x)
 \subset \{(y_1,y'):x_1-\gamma r<y_1\}\cap B_r(x).
$$
\end{assumption}

We remark that the boundary of a Reifenberg flat domain may have a fractal structure. In particular, if the boundary $\partial \Omega$ is locally the graph of a Lipschitz continuous function with a small Lipschitz constant, then $\Omega$ is Reifenberg flat. Thus all $C^1$ domains are Reifenberg flat for any $\gamma>0$.

The next theorem is the main result of this section.

\begin{theorem}
                                \label{thm4}
Let $T\in (0,\infty]$, $\Omega$ be a domain in $\bR^d$, $p \in (1,\infty)$,
and $\vf_\alpha= (f_\alpha^1, \ldots, f_\alpha^n)^{\text{tr}} \in L_p(\Omega_T)$, $|\alpha|\le m$.
Then there exists a constant $\gamma=\gamma(d,n,m,p,\delta)$
such that, under Assumption \ref{assump1} ($\gamma$),
the following hold true.

\noindent
(i)
For any $\vu \in \mathring{\cH}^{m}_p(\Omega_T)$
satisfying
\begin{equation}
                                    \label{eq09_01}
\vu_t+(-1)^m \cL \vu +\lambda \vu = \sum_{|\alpha|\le m}D^\alpha \vf_\alpha
\end{equation}
in $\Omega_T$, we have
\begin{equation}
                                    \label{eq20.39}
\sum_{|\alpha|\le m}\lambda^{1-\frac {|\alpha|} {2m}} \|D^\alpha \vu \|_{L_p(\Omega_T)}
\le N \sum_{|\alpha|\le m}\lambda^{\frac {|\alpha|} {2m}} \| \vf_\alpha \|_{L_p(\Omega_T)},
\end{equation}
provided that $\lambda \ge \lambda_0$,
where $N$ and $\lambda_0 \ge 0$
depend only on $d$, $n$, $m$, $p$, $\delta$, $K$, and $R_0$.

\noindent
(ii) For any  $\lambda > \lambda_0$, there exists a unique $\vu \in \mathring{\cH}_p^{m}(\Omega_T)$ satisfying \eqref{eq09_01}.

\noindent
(iii) Let $\lambda=0$ and suppose $T\in (0,\infty)$. There exists a unique solution $\vu \in \mathring{\cH}_p^{m}((0,T) \times\Omega)$ of \eqref{eq09_01} with the initial condition $\vu(0,\cdot)\equiv 0$ in $\Omega$. Moreover, $\vu$ satisfies
$$
\|\vu \|_{\cH^{m}_p((0,T) \times\Omega)}
\le N \sum_{|\alpha|\le m}\|\vf_\alpha \|_{L_p((0,T) \times\Omega)},
$$
where $N$ depends only on $d$, $n$, $m$, $p$, $\delta$, $K$, $R_0$, and $T$.
\end{theorem}

A similar result for second-order scalar elliptic equations with symmetric coefficient matrices in bounded Reifenberg flat domains was recently studied by Byun and Wang \cite{BW10}. The proofs there are based on an approach developed by the same authors in \cite{BW04}, which in turn uses an idea of approximations originally due to Caffarelli; see \cite{CaPe98}. This approach is quite suitable for studying equations on domains with rough boundaries. Here we adapt it to investigate higher-order systems. Our proofs, however, are in several aspects different from those in \cite{BW10}. In particular, we do not use the reverse H\"older's inequality, the iteration argument or the maximum principle, the latter of which is not available in our case. Compared to \cite{BW10}, we  consider more general operators and domains
by allowing lower-order terms and unbounded domains.

The crucial ingredients of the proofs below are the interior and the boundary estimates established in Sections \ref{sec3.2} and \ref{sec6}.
By a scaling, we may assume $R_0=1$ in the sequel. Recall the definitions of $U$, $V$, $W$, and $F$ in Sections \ref{sec3.2} and \ref{sec4}.
\begin{lemma}
                            \label{lem7.3}
Let $R\in (0,1]$, $\lambda\in (0,\infty)$, $\nu\in (2,\infty)$, $\nu'=2\nu/(\nu-2)$, $\vf_\alpha= (f_\alpha^1, \ldots, f_\alpha^n)^{\text{tr}} \in L_{2,\text{loc}}(\bR\times \Omega)$, $|\alpha|\le m$. Assume that $a_{\alpha\beta}\equiv 0$ for any $\alpha,\beta$ satisfying $|\alpha|+|\beta|<2m$ and that $\vu\in C_0^{\infty}(\bR\times \Omega)$ satisfies \eqref{eq09_01}.
Then the following hold true.

\noindent(i) Suppose $0\in \Omega$, $\dist(0,\partial\Omega)\ge R$, and Assumption \ref{assump1} ($\gamma$) (i) holds at the origin. Then $\vu$ admits a decomposition $\vu=\vv+\vw$ in $Q_{R}$, and $\vw$ and $\vv$ satisfy
\begin{equation}	
 			\label{eq16.15}
(W^2)_{Q_{R}}^{1/2}
\le N\gamma^{1/\nu'} (U^\nu)_{Q_{R}}^{1/\nu}+
N(F^2)_{Q_{R}}^{1/2}
\end{equation}
and
\begin{equation}
                                \label{eq16.22}
\|V\|_{L_\infty(Q_{R/4})}
\le N\gamma^{1/\nu'} (U^\nu)_{Q_{R}}^{1/ \nu}+
N(F^2)_{Q_{R}}^{1/2}+
N(U^2)_{Q_{R}}^{1/2},
\end{equation}
where $N=N(d,n,m,\delta,\nu)>0$ is a constant.

\noindent(ii) Suppose $0\in \partial\Omega$ and Assumption \ref{assump1} ($\gamma$) (ii) holds at the origin. Then $\vu$ admits a decomposition $\vu=\vv+\vw$ in $\cC_{R}$, and $\vw$ and $\vv$ satisfy
\begin{equation}	
 			\label{eq17.10}
(W^2)_{\cC_{R}}^{1/2}
\le N\gamma^{1 /\nu'} (U^\nu)_{\cC_{R}}^{1/\nu}+
N(F^2)_{\cC_{R}}^{1/2}
\end{equation}
and
\begin{equation} \label{eq17.11}
\|V\|_{L_\infty(\cC_{R/4})} \le N\gamma^{1/\nu'}
(U^\nu)_{\cC_{R}}^{1/\nu}
+N(F^2)_{\cC_{R}}^{1/2}+ N(U^2)_{\cC_{R}}^{1/2},
\end{equation}
where $N=N(d,n,m,\delta,\nu)>0$ is a constant.
\end{lemma}
\begin{proof}
The proof is similar to that of Proposition \ref{prop7.9} with some modifications.
Without loss of generality, we may assume Assumption \ref{assump1} holds in the original $(t,x)$-coordinates.
First we prove Assertion (i). Define
$$
\bar a_{\alpha\beta}(y_1)=\dashint_{Q'_R} a_{\alpha\beta}(\tau,y_1,z') \, dz'\,d\tau,
$$
and $\cL_0$ be the operator with $\bar a_{\alpha\beta}$ in place of $a_{\alpha\beta}$ in $\cL$. Then in $Q_R$, $\vu$ satisfies
$$
\vu_t+(-1)^m \cL_0 \vu+\lambda \vu
$$
\begin{equation}
                                    \label{eq17.23}
=(-1)^m \sum_{|\alpha|=|\beta|=m}D^\alpha\left(
(\bar a_{\alpha\beta}- a_{\alpha\beta})D^\beta \vu\right)+ \sum_{|\alpha|\le m}D^\alpha \vf_\alpha.
\end{equation}
Now let $\vw$ be the unique $\cH^{m}_2(\bR^{d+1})$ solution of
$$
\vw_t+(-1)^m \cL_0 \vw+\lambda \vw
$$
$$
=(-1)^m \sum_{|\alpha|=|\beta|=m}D^\alpha\left(
\varphi(\bar a_{\alpha\beta}- a_{\alpha\beta})D^\beta \vu\right)+ \sum_{|\alpha|\le m}D^\alpha (\varphi\vf_\alpha)
$$
in $\bR^{d+1}$, where
$\varphi$ is an infinitely differentiable function such that
$$
0 \le \varphi \le 1,
\quad
\varphi = 1 \,\, \text{on} \,\, Q_{R/2},
\quad
\varphi = 0 \,\, \text{outside} \,\, (-R^{2m}, R^{2m}) \times B_R.
$$
The existence of such solution is due to Theorem \ref{theorem08061901}. By the same theorem, we have
$$
\|W\|_{L_2(\bR^d_0)}
\le N \sum_{|\alpha|\le m}\lambda^{\frac {|\alpha|} {2m}-\frac 1 2} \| \varphi\vf_\alpha \|_{L_2(\bR^d_0)}
$$
$$
+N\sum_{|\alpha|=|\beta|=m}\left\|
\varphi(\bar a_{\alpha\beta}- a_{\alpha\beta})D^\beta \vu\right\|_{L_2(\bR^d_0)},
$$
where $\bR^d_0=(-\infty,0)\times\bR^d$.
This together with H\"older's inequality gives \eqref{eq16.15}. Next, it is easily seen that $\vv:=\vu-\vw$ satisfies in $Q_{R/2}$
\begin{equation}
                            \label{eq17.02}
\vv_t+(-1)^m \cL_0 \vv+\lambda \vv=0.
\end{equation}
With a scaling, we apply Corollary \ref{cor3.10} to \eqref{eq17.02} to get
$$
\|V\|_{L_\infty(Q_{R/4})}
\le N(V^2)_{Q_{R/2}}^{1/2}.
$$
The estimate above together with \eqref{eq16.15} leads to \eqref{eq16.22}.

Next, we prove Assertion (ii). Define $\bar a_{\alpha\beta}$ as above and notice that $\vu$ satisfies \eqref{eq17.23} in $\cC_R$. To apply Corollary \ref{cor3.10b}, we need to locally cutoff $\vu$ so that the new function vanishes for $y_1<\gamma R$. To this end, take the function $\chi$ from the proof of Proposition \ref{prop7.9}.
Then $\hat \vu:=\chi\vu$ along with all its derivatives vanishes on $y_1\le \gamma R$ and satisfies \eqref{eq17.23b} in $Q_R^{\gamma+}:=Q_R\cap \{y_1> \gamma R\}$.
Let $\hat\vw$ be the unique $\mathring\cH^{m}_2(\bR\times \{y:y_1>\gamma R\})$ solution of \eqref{eq17.25} in $\bR\times \{y:y_1>\gamma R\}$.
Following the argument in the proof of Proposition \ref{prop7.9} with obvious modifications, we get
\begin{equation}	
 			\label{eq21.52}
\sum_{k=0}^m\lambda^{\frac 1 2-\frac k {2m}}(I_{Q_{R}^{\gamma +}}|D^k \hat \vw|^2)_{\cC_R}^{1/2}
\le N\gamma^{1/ {\nu'}} (U^\nu)_{\cC_{R}}^{1/ \nu}+
N(F^2)_{\cC_{R}}^{1/2}.
\end{equation}
We extend $\hat \vw$ to be zero in $\cC_R\setminus Q_R^{\gamma +}$, so that $\hat\vw\in \cH^{m}_2(\cC_R)$, and let $\vw=\hat \vw+(1-\chi)\vu$.
By Lemma \ref{gHardy} and H\"older's inequality, we get \eqref{eq17.10} from \eqref{eq21.52}.

Next, we define $\vv=\vu-\vw$ in $\cC_R$. It is easily seen that $\vv=0$ in $\cC_R\setminus Q_R^{\gamma +}$ and $\vv$ satisfies
$$
\vv_t+(-1)^m \cL_0 \vv+\lambda \vv=0
$$
in  $Q_{R/2} \cap \{y_1> \gamma R\}$ and vanishes along with its derivatives up to $(m-1)$-th order on $Q_R \cap \{y_1=\gamma R\}$. Applying Corollary \ref{cor3.10b}, we get
$$
\|V\|_{L_\infty(\cC_{R/4})}=
\|V\|_{L_\infty(Q_{R/4}\cap\{y_1>\gamma R\})}
\le N(V^2)_{\cC_{R/2}}^{1/2},
$$
which together with \eqref{eq17.10} gives \eqref{eq17.11}.
This completes the proof of the lemma.
\end{proof}

For a function $f$ on a set $\cD\subset\bR^{d+1}$, we define
its maximal function $\cM f$ by $\cM f=\cM (I_{\cD}f)$.
For any $s>0$, we introduce two level sets
$$
\cA(s)=\{(t,x)\in \bR\times\Omega:U> s\},
$$
$$
\cB(s)=\Big\{(t,x)\in \bR\times\Omega:\gamma^{-1/\nu'}\big(\cM(F^2)\big)^{1/2}+
\big(\cM(U^\nu)\big)^{1/\nu}>s\Big\}.
$$
With Lemma \ref{lem7.3} in hand, we get the following corollary.

\begin{corollary}
                                    \label{cor7.5}
Under the assumptions of Lemma \ref{lem7.3}, suppose $0\in \bar\Omega$ and Assumption \ref{assump1} ($\gamma$) holds. Let $s\in (0,\infty)$ be a constant. Then there exists a constant $\kappa\in (1,\infty)$, depending only on $d$, $n$, $m$, $\delta$, and $\nu$, such that the following holds.
If
\begin{equation}
                    \label{eq15.59}
\big|\cC_{R/32}\cap \cA(\kappa s)\big|> \gamma^{2/\nu'} |\cC_{R/32}|,
\end{equation}
then we have $\cC_{R/32}\subset \cB(s)$.
\end{corollary}
\begin{proof}
By dividing $\vu$ and $\vf$ by $s$, we may assume $s=1$.
We prove by contradiction. Suppose at a point $(t,x)\in \cC_{R/32}$, we have
\begin{equation}
                            \label{eq15.38}
\gamma^{-1/\nu'}\big(\cM(F^2)(t,x)\big)^{1/2}+
\big(\cM(U^\nu)(t,x)\big)^{1/\nu}\le 1.
\end{equation}
Let us consider two cases.

{\em Case 1: $\dist(0,\partial \Omega)\ge R/8$.}
Notice that
$$
(t,x)\in\cC_{R/32}=Q_{R/32}\subset Q_{R/8}\subset \bR\times \Omega.
$$
Due to Lemma \ref{lem7.3} (i), we can write $\vu=\vw+\vv$ in $Q_{R/8}$ and, by \eqref{eq15.38},
\begin{equation}
                                \label{eq15.47}
\|V\|_{L_\infty(Q_{R/32})}
\le N_1,
\quad
(W^2)_{Q_{R/8}}^{1/2}
\le N_1\gamma^{1/\nu'},
\end{equation}
where $N_1$ and constants $N_i$ below depend only on $d$, $n$, $m$, $\delta$, and $\nu$. By \eqref{eq15.47}, the triangle inequality and Chebyshev's inequality, we get
$$
\big|\cC_{R/32}\cap \cA(\kappa)\big|=
\big|\{(t,x)\in \cC_{R/32}: U>\kappa\}\big|
$$
$$
\le \big|\{(t,x)\in \cC_{R/32}: W>\kappa-N_1\}\big|
\le (\kappa-N_1)^{-2} N_1^2 \gamma^{2/\nu'}|Q_{R/8}|,
$$
which contradicts with \eqref{eq15.59} if we choose $\kappa$ sufficiently large.

{\em Case 2: $\dist(0,\partial \Omega)< R/8$.} We take $y\in \partial\Omega$ such that $|y|=\dist(0,\partial \Omega)$. Notice that in this case we have
$$
(t,x)\in\cC_{R/32}\subset \cC_{R/4}(0,y)\subset \cC_{R}(0,y).
$$
Due to Lemma \ref{lem7.3} (ii), we can write $\vu=\vw+\vv$ in $\cC_R(0,y)$ and, by \eqref{eq15.38},
\begin{equation}
                                \label{eq17.23c}
\|V\|_{L_\infty(\cC_{R/4}(0,y))}
\le N_2,
\quad
(W^2)_{\cC_{R}(0,y)}^{1/2}
\le N_2\gamma^{1/\nu'}.
\end{equation}
By \eqref{eq17.23c}, the triangle inequality and Chebyshev's inequality, we get
$$
\big|\cC_{R/32}\cap \cA(\kappa)\big|=
\big|\{(t,x)\in \cC_{R/32}: U> \kappa\}\big|
$$
$$
\le \big|\{(t,x)\in \cC_{R/32}: W> \kappa-N_2\}\big|
\le (\kappa-N_2)^{-2} N_2^2 \gamma^{2/\nu'}|\cC_{R}|,
$$
which contradicts with \eqref{eq15.59} if we choose $\kappa$ sufficiently large.
\end{proof}

\begin{theorem}							\label{theorem101}
Let $p \in (2,\infty)$, $\lambda > 0$, $X_0\in \bR^{d+1}$
and $\vf_{\alpha} \in L_p(\bR\times \Omega)$. Suppose that $a_{\alpha\beta}\equiv 0$ for any $\alpha,\beta$ satisfying $|\alpha|+|\beta|<2m$, and $\vu \in C_0^{\infty}(\bR\times \Omega)$ vanishes outside $Q_{\gamma}(X_0)$ and satisfies \eqref{eq09_01} in $\bR\times \Omega$.
There exist positive constants $\gamma \in (0,1/20)$ and $N$,
depending only on $d,\delta,m,n, p$, such that, under Assumption \ref{assump1} ($\gamma$)
we have
\begin{equation}
                                \label{eq16.00}
\|U \|_{L_p(\bR\times \Omega)}
\le N \|F\|_{L_p(\bR\times \Omega)},
\end{equation}
where $N = N(d,\delta,m,n,p)$.
\end{theorem}

\begin{proof}
We fix $\nu=p/2+1$ and let $\nu'=2\nu/(\nu-2)$. Let $\kappa$ be the constant in Corollary \ref{cor7.5}. For any $s>0$, by Chebyshev's inequality,
\begin{equation}
                            \label{eq17.57}
|\cA(\kappa s)|\le (\kappa s)^{-2}\|U\|_{L_2(\bR\times\Omega)}^2.
\end{equation}
From \eqref{eq17.57}, Corollary \ref{cor7.5}, and a result from measure theory
on the ``crawling of ink spots'' which can be found in \cite{Sa80} or \cite[Sect. 2]{KS80}, we have the following upper bound of the distribution of $U$ (cf. \cite[Sect. 4]{BW04}). For any $\gamma\in (0, 1/20]$ and $\kappa s\ge \gamma^{-1/\nu'}\|U\|_{L_2(\bR\times\Omega)}$,
\begin{equation}
                            \label{eq18.34}
|\cA(\kappa s)|\le  N_4\gamma^{2/\nu'}|\cB(s)|.
\end{equation}
For $0<\kappa s<\gamma^{-1/\nu'}\|U\|_{L_2(\bR\times\Omega)}$, we simply bound the distribution function by \eqref{eq17.57}.
Now we use the elementary identity:
$$
\|f\|_{L_p(\cD)}^p=p\int_0^\infty \big|\{(t,x)\in\cD:|f(t,x)|> s\}\big|s^{p-1}\,ds,
$$
to deduce from \eqref{eq17.57} and \eqref{eq18.34} that
$$
\|U\|_{L_p(\bR\times\Omega)}^p \le N_5\gamma^{(2-p)/\nu'}\Big(\|U\|_{L_2(\bR\times\Omega)}^p+
\big\|\big(\cM(F^2)\big)^{1/2}\big\|_{L_p(\bR\times\Omega)}^p\Big)
$$
$$
+N_5\gamma^{2/\nu'}\big\|\big(\cM(U^\nu)\big)^{1/\nu}
\big\|_{L_p(\bR\times\Omega)}^p.
$$
By H\"older's inequality,
\begin{equation}
                            \label{eq20.31bb}
\|U\|_{L_2(\bR\times\Omega)}=\|U\|_{L_2(Q_\gamma(X_0)\cap\bR\times\Omega)}
\le N\|U\|_{L_p(\bR\times\Omega)}\gamma^{(d+2m)(1/2-1/p)}.
\end{equation}
Since $2<\nu<p$, by the Hardy--Littlewood maximal function theorem and \eqref{eq20.31bb}, we obtain
$$
\|U\|_{L_p(\bR\times\Omega)}^p \le N_6\gamma^{(2-p)/\nu'}
\|F\|_{L_p(\bR\times\Omega)}^p
+N_6\gamma^{2/\nu'}\|U\|_{L_p(\bR\times\Omega)}^p.
$$
To get the estimate \eqref{eq16.00}, it suffices to take $\gamma=\gamma(d,n,m,\delta,p)\in (0,1/20]$ sufficiently small such that $N_6\gamma^{2/\nu'}\le 1/2$.
\end{proof}

Now we are ready to give the proof of Theorem \ref{thm4}.

\begin{proof}[Proof of Theorem \ref{thm4}]
We derive Assertion (i) from Theorem \ref{theorem101} as in the proof of Theorem \ref{thm1}. Assertion (ii) follows from Assertion (i), the method of continuity and an approximation argument. To prove Assertion (iii), let $\vu=e^{t(\lambda_0+1)}\tilde \vu$. Clearly $\tilde \vu$ satisfies \eqref{eq09_01} with $\lambda_0+1$ and $\tilde {\vf_\alpha}:=e^{-t(\lambda_0+1)}\vf_\alpha$ in place of $\lambda$ $(=0)$ and $\vf$ respectively.
Therefore, to show the existence and uniqueness of $\vu$ is equivalent to show those of the solution $\tilde \vu$ to the equation
\begin{equation}
                    \label{eq18.07}
\vv_t+(-1)^m \cL \vv +(\lambda_0+1) \vv = \sum_{|\alpha|\le m}D^\alpha \tilde{\vf_\alpha}
\end{equation}
with the zero initial condition. We extend $\tilde \vf$ to be zero for $t\le 0$. Using Assertion (ii), we find a unique $\mathring{\cH}^m_p(\Omega_T)$ solution $\vv$ of \eqref{eq18.07} in $\Omega_T$.
Set $\tilde \vu := \vv$.
By the uniqueness, $\vv\equiv 0$ for $t\le 0$, which implies that $\tilde \vu$ satisfies the same equation with the zero initial condition at $t=0$. This gives the existence. For the uniqueness, let $\tilde \vu\in \mathring{\cH}^m_p((0,T)\Omega)$ be a solution to \eqref{eq18.07} with the zero initial condition and zero right-hand side. We extend $\tilde \vu$ to be zero for $t\le 0$ and denote it by $\vv$. It is easily seen that $\vv$ satisfies \eqref{eq18.07} in $\Omega_T$ with zero right-hand side. By Assertion (ii), we have $\vv\equiv 0$ in $\Omega_T$. Finally, the estimate of $\vu$ is deduced from the estimate of $\tilde \vu$, which in turn follows from Assertion (ii). The theorem is proved.
\end{proof}

We finish the paper by considering the corresponding Dirichlet boundary value problem for elliptic equations in the case that all the involved functions are independent of the time variable.

\begin{theorem}
                                \label{thm5}
Let $\Omega$ be a domain in $\bR^d$, $p \in (1,\infty)$,
and $\vf_\alpha= (f_\alpha^1, \ldots, f_\alpha^n)^{\text{tr}} \in L_p(\Omega)$, $|\alpha|\le m$.
Then there exists a constant $\gamma=\gamma(d,n,m,p,\delta)$
such that, under Assumption \ref{assump1} ($\gamma$),
the following hold true.

\noindent
(i)
For any $\vu \in \mathring{W}^m_p(\Omega)$
satisfying
\begin{equation}
                                    \label{eq11_01}
\cL \vu +(-1)^m \lambda \vu = \sum_{|\alpha|\le m}D^\alpha \vf_\alpha \quad \text{in}\quad \Omega,
\end{equation}
we have
\begin{equation*}
\sum_{|\alpha|\le m}\lambda^{1-\frac {|\alpha|} {2m}} \|D^\alpha \vu \|_{L_p(\Omega)}
\le N \sum_{|\alpha|\le m}\lambda^{\frac {|\alpha|} {2m}} \| \vf_\alpha \|_{L_p(\Omega)},
\end{equation*}
provided that $\lambda \ge \lambda_0$,
where $N$ and $\lambda_0 \ge 0$
depend only on $d$, $n$, $m$, $p$, $\delta$, $K$, and $R_0$.

\noindent
(ii) For any  $\lambda > \lambda_0$, there exists a unique $\vu \in \mathring{W}_p^m(\Omega)$ satisfying \eqref{eq11_01}.

\noindent
(iii) Let $\lambda=0$. Assume that $|\Omega|<\infty$ and that equation \eqref{eq11_01} admits a $W^m_2$ estimate, i.e. for any $\vu\in \mathring{W}^m_2(\Omega)$ satisfying \eqref{eq11_01} with $\vf_\alpha\in L_2(\Omega)$, we have
$$
\|\vu \|_{W^{m}_2(\Omega)}
\le N \sum_{|\alpha|\le m}\|\vf_\alpha \|_{L_2(\Omega)},
$$
where $N$ depends only on $d$, $n$, $m$, $\delta$, $K$, $R_0$, and $|\Omega|$.
Then there exists a unique solution $\vu \in \mathring{W}_p^m(\Omega)$ of \eqref{eq11_01} and $\vu$ satisfies
\begin{equation}
                                \label{eq23.08}
\|\vu \|_{W^{m}_p(\Omega)}
\le N \sum_{|\alpha|\le m}\|\vf_\alpha \|_{L_p(\Omega)},
\end{equation}
where $N$ depends only on $d$, $n$, $m$, $p$, $\delta$, $K$, $R_0$, and $|\Omega|$.
\end{theorem}
\begin{proof}
The proofs of the first two assertions in Theorem \ref{thm5}
are completely analogous to those of Theorem \ref{thm4}, and are thus omitted.

For the last assertion, as before it is sufficient to show the a priori estimate \eqref{eq23.08} for $\vu\in \mathring{W}_p^m(\Omega)$. Again, by a duality argument we may focus on the case $p>2$. From Assertion (i), we have
\begin{equation}
                                \label{eq23.17}
\|\vu \|_{W^{m}_p(\Omega)}
\le N \sum_{|\alpha|\le m}\|\vf_\alpha \|_{L_p(\Omega)}+N_7\|\vu \|_{L_p(\Omega)},
\end{equation}
Therefore, we only need to estimate $\|\vu\|_{L_p(\Omega)}$.
Take $p_1\in (p,\infty)$ such that  $1-d/p>-d/p_1$. By H\"older's inequality, Young's inequality and the Poincar\'e-Sobolev inequality, we get for any $\varepsilon>0$,
$$
\|\vu\|_{L_p(\Omega)}\le N(\varepsilon)\|\vu\|_{L_2(\Omega)}+ \varepsilon\|\vu\|_{L_{p_1}(\Omega)}\le N(\varepsilon)\|\vu\|_{L_2(\Omega)}+N_8\varepsilon\|D\vu\|_{L_{p}(\Omega)}.
$$
Choosing $\varepsilon=1/(2N_7N_8)$ and using \eqref{eq23.17}, we obtain \eqref{eq23.08}. The theorem is proved.
\end{proof}

\begin{remark}
As an immediate corollary of Theorem \ref{thm5}, we obtain the $\mathring{W}^m_p(\Omega)$ solvability of \eqref{eq11_01} with $\lambda=0$ if, in addition, $|\Omega|<\infty$ and the lower-order coefficients of $\cL$ are all zero. Finally, for second-order scalar elliptic equations in the form
$$
D_j(a_{ij}D_i u)+D_i(a_i u)+b_iD_i u+cu=\Div g+f,
$$
under the same assumptions and the additional condition that $D_ia_i+c\le 0$ in $\Omega$ in the weak sense, we also get the $\mathring W^1_p(\Omega)$ solvability by using the classical $W^1_2$ estimate of the Dirichlet problem on a domain with a bounded measure; see \cite[\S 8]{GT83} or \cite[Sect. 7]{DongKim08a}. This result generalizes Theorem 2.7 in \cite{DongKim08a}, in which bounded Lipschitz domains are considered. It also extends the main result of \cite{BW10} to equations with lower-order terms.
\end{remark}



\bibliographystyle{plain}

\def\cprime{$'$}\def\cprime{$'$} \def\cprime{$'$} \def\cprime{$'$}
  \def\cprime{$'$} \def\cprime{$'$}

\end{document}